\pdfoutput=1
\documentclass[11pt]{article}
\usepackage{amssymb}
\usepackage[all,cmtip]{xy}
\usepackage[width=17cm,top = 3.5cm,bottom=3.5cm]{geometry}
\usepackage{tikz}
\usepackage[all]{xy}
\usetikzlibrary{matrix,arrows,decorations.pathmorphing}
\usepackage{amsthm,amsfonts,amssymb,amsmath,amscd} 
\usepackage{aliascnt} 
\usepackage{mathrsfs} 
\usepackage{xargs,ifthen} 

\usepackage{comment}
\usepackage{color}

\usepackage{hyperref}
\definecolor{tocolor}{rgb}{.1,.1,.1}
\definecolor{urlcolor}{rgb}{.2,.2,.6}
\definecolor{linkcolor}{rgb}{.1,.1,.5}
\definecolor{citecolor}{rgb}{.4,.2,.1}
\hypersetup{backref=true, colorlinks=true, urlcolor=urlcolor, linkcolor=linkcolor, citecolor=citecolor}

\newcommandx{\thdef}[2]{
	\newaliascnt{#1}{theorem}  
	\newtheorem{#1}[#1]{#2}
	\aliascntresetthe{#1}  
	\newtheorem*{#1*}{#2}
	\expandafter\newcommand\expandafter{\csname #1autorefname\endcsname}{#2}
}

\makeatletter
\newtheorem*{rep@theorem}{\rep@title}
\newcommand{\newreptheorem}[2]{%
\newenvironment{rep#1}[1]{%
 \def\rep@title{#2 \ref{##1}}%
 \begin{rep@theorem}}%
 {\end{rep@theorem}}}
\makeatother

\newtheorem{theorem}{Theorem}[section]
\newreptheorem{theorem}{Theorem}

\thdef{lemma}{Lemma}
\thdef{corollary}{Corollary}
\thdef{conjecture}{Conjecture}
\newreptheorem{conjecture}{Conjecture}
\thdef{proposition}{Proposition}
\thdef{question}{Question}

\theoremstyle{definition}
\thdef{definition}{Definition}
\thdef{notation}{Notation}

\theoremstyle{remark}
\thdef{remark}{Remark}

\theoremstyle{remark}
\thdef{ex}{Example}

\newenvironment{example}
{\begin{ex}}%
{\hfill $\blacksquare$\end{ex}}

\newcommand{\spc}[1]{\mathsf{#1}} 
\newcommand{\shf}[1]{\mathcal{#1}} 

\newcommand{\rbrac}[1]{\left(#1\right)} 

\newcommandx{\fn}[2][2=]{#1\ifthenelse{\equal{#2}{}}{}{\!\rbrac{{#2}}}} 
\newcommandx{\id}[2][2=]{\fn{{\rm id}_{#1}}[#2]} 

\newcommand{\ext}[2][\bullet]{\spc{\Lambda}^{#1}{#2}} 
\newcommandx{\End}[2][1=]{\fn{\spc{End}_{#1}}[#2]} 
\newcommandx{\Hom}[2][1=]{\fn{\spc{Hom}_{#1}}[#2]} 
\newcommandx{\Aut}[2][1=]{\fn{\spc{Aut}_{#1}}[#2]} 
\newcommandx{\image}[1]{\fn{\spc{img}}[#1]} 
\renewcommandx{\ker}[1]{\fn{\spc{ker}}[#1]} 
\newcommandx{\rank}[1]{\fn{\mathrm{rank}}[#1]} 

\newcommandx{\ann}[1]{\fn{\spc{ann}}[#1]} 

\newcommandx{\hlgy}[3][1=\bullet,3=]{\spc{H}_{#1}^{#3}\!\rbrac{{#2}}} 
\newcommandx{\cohlgy}[3][1=\bullet,3=]{\spc{H}^{#1}_{#3}\!\rbrac{{#2}}} 
\newcommandx{\chow}[3][1=\bullet,3=]{\spc{A}^{#1}_{#3}\!\rbrac{{#2}}} 

\newcommandx{\Ext}[3][1=\bullet,3=]{\fn{\spc{Ext}^{#1}_{#3}}[{#2}]} 
\newcommandx{\Tor}[3][1=\bullet,3=]{\fn{\spc{Tor}^{#1}_{#3}}[{#2}]} 

\newcommandx{\Pic}[1]{\fn{\spc{Pic}}[{#1}]} 
\newcommandx{\chernalg}[2][1=\bullet]{\fn{\spc{Chern}^{#1}}[{#2}]} 
\newcommandx{\chern}[2][1=]{\fn{c_{#1}}[#2]} 
\newcommandx{\ch}[2][1=]{\fn{\mathrm{ch}_{#1}}[{#2}]} 

\newcommandx{\sKer}[2][1=]{ \fn{ \shf{K}er_{#1}}[{#2}] } 
\newcommandx{\sHom}[2][1=]{ \fn{ \shf{H}om_{#1}}[{#2}] } 
\newcommandx{\sEnd}[2][1=]{ \fn{ \shf{E}nd_{#1}}[{#2}] } 
\newcommandx{\sExt}[3][1=\bullet,3=]{\fn{\shf{E}xt^{#1}_{#3}}[{#2}]} 
\newcommandx{\sTor}[3][1=\bullet,3=]{\fn{\shf{T}or^{#1}_{#3}}[{#2}]} 

\newcommandx{\forms}[2][1=\bullet]{\Omega^{#1}_{#2}} 
\newcommandx{\can}[1][1=]{\omega_{#1}} 
\newcommandx{\acan}[1][1=]{\omega_{#1}^{-1}} 
\newcommandx{\tshf}[1]{\shf{T}_{#1}} 
\newcommandx{\mvect}[2][1=\bullet]{ \ext[#1]{\tshf{#2}} }
\newcommandx{\der}[2][1=\bullet]{\mathscr{X}^{#1}_{#2}} 
\newcommandx{\sJet}[3][1=,2=]{\shf{J}^{#1}_{#2}#3} 

\newcommandx{\tb}[2][1=]{\spc{T}_{\!#1}{#2}} 
\newcommandx{\ctb}[2][1=]{\spc{T}_{\!#1}^*{#2}} 
\newcommandx{\lie}[2][2=]{\fn{\mathscr{L}_{#1}}[#2]} 
\newcommandx{\hook}[2][2=]{\fn{i_{#1}}[#2]} 

\makeatletter
\newcommand{\thickbar}{\mathpalette\@thickbar}
\newcommand{\@thickbar}[2]{{#1\mkern1.5mu\vbox{
  \sbox\z@{$#1\mkern-1mu#2\mkern-1mu$}%
  \sbox\tw@{$#1\overline{#2}$}%
  \dimen@=\dimexpr\ht\tw@-\ht\z@-.6\p@\relax
  \hrule\@height.4\p@ 
  \vskip1\p@
  \hrule\@height.4\p@ 
  \vskip\dimen@
  \box\z@}\mkern1.5mu}
}
\makeatother

\def\ker{\text{ker}}
\def\im{\text{im}}
\def\End{\text{End}}
\def\log{\text{log}}

\renewcommand{\cal}[1]{\mathcal{#1}}

\numberwithin{equation}{section}
\newtheoremstyle{parag}
  {\topsep}   
  {\topsep}   
  {}  
  {}       
  {\bfseries} 
  {.}         
  { } 
  {}          
\theoremstyle{parag}

\usepackage{MnSymbol}

\makeatletter
\def\@cite#1#2{{\normalfont[{#1\if@tempswa , #2\fi}]}}
\makeatother

\renewcommand{\Pic}{\mathrm{Pic}}

\def\Aut{\text{Aut}}

\begin{document}

\title{\vspace{-4em} \huge The derived moduli stack of logarithmic flat connections}

\date{}

\author{
Francis Bischoff\thanks{Exeter College and Mathematical Institute, University of Oxford; {\tt francis.bischoff@maths.ox.ac.uk }}
}
\maketitle
\abstract{We give an explicit finite-dimensional model for the derived moduli stack of flat connections on $\mathbb{C}^k$ with logarithmic singularities along a weighted homogeneous Saito free divisor. We investigate in detail the case of plane curves of the form $x^p = y^q$ and relate the moduli spaces to the Grothendieck-Springer resolution. We also discuss the shifted Poisson geometry of these moduli spaces. Namely, we conjecture that the map restricting a logarithmic connection to the complement of the divisor admits a shifted coisotropic structure and we construct a shifted Poisson structure on the formal neighbourhood of a canonical connection in the case of plane curves $x^p = y^q$. }

\tableofcontents

\section{Introduction}
Let $D \subset \mathbb{C}^{k}$ be a hypersurface cut out by a reduced holomorphic function $f$. In \cite{saito1980theory} Saito considers the subsheaf, usually denoted $T_{\mathbb{C}^k}(-\log D)$, of holomorphic vector fields on $\mathbb{C}^{k}$ which preserve the ideal generated by $f$. In general, it is coherent and closed under the Lie bracket, but may fail to be locally free. In fact, Saito provides a very explicit criterion for determining whether the sheaf is locally free. When it is, $D$ is said to be a free divisor and $T_{\mathbb{C}^k}(-\log D)$, known as the logarithmic tangent bundle, defines a Lie algebroid. Examples of free divisors include smooth hypersurfaces, plane curves and simple normal crossings. In general, $D$ may be highly singular. 

Let $G$ be a connected complex reductive group with Lie algebra $\mathfrak{g}$ and assume that $D$ is a free divisor which is homogeneous under a given $\mathbb{C}^*$-action on $\mathbb{C}^k$ with the property that all its weights are strictly positive. In this paper, we are interested in studying the moduli space of $T_{\mathbb{C}^k}(-\log D)$-representations on principal $G$-bundles, also known as \emph{logarithmic flat connections}. There is a standard way of defining this moduli space as the Maurer-Cartan locus of an infinite-dimensional differential graded Lie algebra (dgla) $L_{D, \mathfrak{g}}$ which is associated to $D$ and $\mathfrak{g}$. Let $\Omega^{1}_{\mathbb{C}^{k}}(\log D)$ denote the logarithmic cotangent bundle, which is the dual to $T_{\mathbb{C}^k}(-\log D)$, and let $\Omega^{\bullet}_{\mathbb{C}^{k}}(\log D) = \wedge^{\bullet} \Omega^{1}_{\mathbb{C}^{k}}(\log D)$ be the exterior algebra. This defines a commutative differential graded algebra when equipped with the Lie algebroid differential $d$. Then $L_{D,\mathfrak{g}} = \Omega^{\bullet}_{\mathbb{C}^{k}}(\log D) \otimes \mathfrak{g}$ inherits the structure of a dgla. The Maurer-Cartan locus of this dgla is defined to be the following set
\[
MC(L_{D, \mathfrak{g}}) = \{ \omega \in L^{1}_{D,\mathfrak{g}} \ | \ d\omega + \frac{1}{2}[\omega, \omega] = 0 \}. 
\]
Here, $\omega \in \Omega^{1}_{\mathbb{C}^{k}}(\log D) \otimes \mathfrak{g}$ is a Lie algebra valued $1$-form, and it defines the following connection $\nabla = d + \omega$, which has a logarithmic singularity along $D$. It's curvature is given by the following expression
\[
F(\omega) = d\omega + \frac{1}{2}[\omega, \omega]. 
\]
The degree $0$ component of the dgla is $L^{0}_{D,\mathfrak{g}} = Map(\mathbb{C}^{k}, \mathfrak{g})$, which is the Lie algebra of the infinite dimensional gauge group $\mathfrak{G} = Map(\mathbb{C}^{k}, G)$. This group acts on the Maurer-Cartan locus, giving the correct equivalence between flat connections. As a result, the moduli space of flat logarithmic connections is defined to be the stack quotient 
\[
[MC(L_{D, \mathfrak{g}}) /\mathfrak{G} ]. 
\]
Although this construction involves infinite dimensional spaces, in \cite{bischoff2022normal} we provide a purely finite dimensional model. More precisely, we show that the category of logarithmic flat connections with fixed residue data is equivalent to the stack quotient of an affine algebraic variety by the action of an algebraic group. 

The purpose of the present paper is to provide a derived enhancement of the moduli stack. There are several different approaches to derived geometry in the literature, such as \cite{MR1801413, MR1839580, lurie2018spectral, MR2394633, MR3073928}. In this paper, we have opted to go with the notion of bundles of curved dgla's, which requires relatively little technology and is sufficient for our purposes. Let us recall the definition from \cite{MR3248985}. 

\begin{definition}
A \emph{bundle of curved differential graded Lie algebras} over a variety $M$ consists of a graded vector bundle $\mathcal{L}^{\bullet}$ starting in degree $2$, which is equipped with the following data 
\begin{enumerate}
\item a section $F \in \Gamma(M, \mathcal{L}^{2})$, 
\item a degree $1$ bundle map $\delta : \mathcal{L}^{\bullet} \to \mathcal{L}^{\bullet}[1]$ 
\item a smoothly varying graded Lie bracket $[ - , - ]$ on the fibres of $\mathcal{L}^{\bullet}$,
\end{enumerate}
satisfying the following conditions
\begin{enumerate}
\item the Bianchi identity $\delta{F} = 0$, 
\item $\delta^2 = [ F, - ]$, 
\item $\delta$ is a graded derivation of the bracket $[ - , - ]$. 
\end{enumerate}
If $[M/\mathcal{G}]$ is a stack, defined by the data of a Lie groupoid $\mathcal{G}$ over $M$, then we define a bundle of curved dgla's over the stack to be such a bundle over $M$, which is equipped with an equivariant action of $\mathcal{G}$ preserving the data $(F, \delta, [ - , - ])$. We will use the respective terminology of \emph{derived manifolds} and \emph{derived stacks} to refer to this data. 
\end{definition}

There is a standard way of constructing a derived manifold from the data of a dgla, and we can apply it to the case of $L_{D, \mathfrak{g}}$. Namely, we take the base to be $M = L_{D, \mathfrak{g}}^{1}$ and take the bundle  $\mathcal{L}^{\bullet}$ to be trivial with fibre given by the truncation $ L_{D, \mathfrak{g}}^{\bullet \geq 2}$. The section is given by the curvature $F$, and the bracket is simply the constant one inherited from $L_{D, \mathfrak{g}}$. The bundle map $\delta$ varies over $M$, and above a point $\omega \in L_{D, \mathfrak{g}}^{1}$, it is given by the twisted differential $\delta_{\omega} = d + [ \omega, - ]$. Let us denote the resulting derived manifold $\mathcal{M}_{D, \mathfrak{g}}$. It can be further upgraded to a derived stack by noting that the gauge group $\mathfrak{G}$ lifts to an action on $ L_{D, \mathfrak{g}}^{\bullet \geq 2}$ via the adjoint representation. 

A derived manifold (or stack) has an underlying classical truncation $\pi_{0}(\mathcal{M})$, defined as the vanishing locus of the section $F$. In the example under consideration, the classical truncation is given by the Maurer-Cartan locus, and hence 
\[
\pi_{0}([\mathcal{M}_{D, \mathfrak{g}}/\mathfrak{G}]) = [MC(L_{D, \mathfrak{g}})/\mathfrak{G}]. 
\]
For this reason, we say that $[\mathcal{M}_{D, \mathfrak{g}}/\mathfrak{G}]$ is the derived moduli stack of flat logarithmic $G$-connections. The main result of this paper is a finite-dimensional model of this derived stack. 

Here is a brief description of this result. Given an element $A \in \mathfrak{g}$, we consider the infinite dimensional derived moduli stack $[\mathcal{M}_{D, \mathfrak{g}}(A)/\mathfrak{G}]$ of $G$-connections whose `residue' is conjugate to $A$. Details of this are given in Section \ref{hfdandlog}. Let $A = S + N_{0}$ be the Jordan decomposition, where $S$ is semisimple and $N_{0}$ is nilpotent. In Section \ref{finitedimmodel} we construct a finite dimensional dgla $(U_{0}, \delta_{S})$ associated to $S$, with corresponding derived stack $[\mathcal{U}_{S}/Aut(S)]$. This is interpreted as a certain sub-moduli space of flat connections on the fibre $f^{-1}(1)$. Then, associated to the element $A$, we construct a derived substack $[\mathcal{W}(A)/Aut(S)]$ of the shifted tangent bundle $T[-1] [\mathcal{U}_{S}/Aut(S)]$. The main result is Theorem \ref{Maintheorem}, which states that there is an equivalence of derived stacks
\[
q: [\mathcal{W}(A)/Aut(S)] \to [\mathcal{M}_{D, \mathfrak{g}}(A)/\mathfrak{G}].
\] 
By this we mean that $q$ induces an equivalence between the groupoids of solutions to the MC equation, and given any solution $w$, the derivative $dq_{w}$ is a quasi-isomorphism of tangent complexes. 

In Section \ref{planecurvesection} we turn to the case of a plane curve defined by the function $f = x^{p} - y^{q}$. This is the simplest case above $k = 1$, and already it exhibits interesting behaviour. We construct an explicit derived stack $[\mathcal{Q}(A)/P_{S}]$ from the data of a parabolic subgroup $P_{S}$ of the centralizer of $exp(\frac{2\pi i}{pq}S)$ and a representation $H^{1}(U_{0})$. In Theorem \ref{parabolicderivedstack} we show that, given a condition on the eigenvalues of $ad_{S}$, the derived stack $[\mathcal{Q}(A)/P_{S}]$ is equivalent to $[\mathcal{W}(A)/Aut(S)]$. For general $S$, the moduli space can have extra components and we illustrate this in Example \ref{Extracomponent}. The derived stack $[\mathcal{Q}(A)/P_{S}]$ can be interpreted in terms of spaces showing up in geometric representation theory, such as the Grothendieck-Springer resolution. Hence Theorem \ref{parabolicderivedstack} can be viewed as a higher dimensional generalization of Boalch's description in \cite{boalch2011riemann} of the moduli space of logarithmic connections on the disc. 

\subsection*{Speculations about Poisson geometry}
Going back to the work of Atiyah-Bott \cite{MR702806} and Goldman \cite{MR762512}, we know that the moduli space of flat connections on a closed Riemann surface admits a symplectic structure. If the surface is punctured, then the moduli space admits a Poisson structure, whose symplectic leaves are obtained by fixing boundary conditions at the punctures \cite{MR1730456}. This picture has since been generalized in several directions, including to the case of flat connections with singularities \cite{MR1864833, MR1904670, MR2352135}. More recently, the moduli space of local systems on higher dimensional manifolds has been studied using tools from derived algebraic geometry. For a compact oriented manifold $M$ of dimension $d$, the moduli space of local systems $Loc_{G}(M)$ is a derived stack equipped with a shifted symplectic structure of degree $2-d$ \cite{MR3090262}. If $M$ has a boundary $\partial M = N$, then $Loc_{G}(N)$ has a $3-d$-shifted symplectic structure, and the restriction map $r: Loc_{G}(M) \to Loc_{G}(N)$ has a Lagrangian structure \cite{MR3381468}, inducing on $Loc_{G}(M)$ a $2-d$-shifted Poisson structure \cite{MR3848017}. We wish to generalize this picture to the case of logarithmic flat connections in higher dimensions. 

In the above setting of a map $f: \mathbb{C}^k \to \mathbb{C}$, the inverse image of the unit circle $f^{-1}(S^{1})$ is a manifold of dimension $2k - 1$, usually with boundary, and so $Loc_{G}(f^{-1}(S^{1}))$ has a Poisson structure of degree $3 - 2k $. Given a logarithmic flat connection, we can restrict it to $f^{-1}(S^{1})$ and take its holonomy. This should define a map 
\begin{equation} \label{restrictionmap}
 r: [\mathcal{W}(A)/Aut(S)] \to Loc_{G}(f^{-1}(S^{1})). 
\end{equation}
This map was studied by Boalch \cite{boalch2011riemann} in the special case of $k = 1$, where $f^{-1}(S^{1}) = S^{1}$. In this case $Loc_{G}(S^{1}) = G/G$ has a $1$-shifted symplectic structure, and the work of Boalch (suitably interpreted by \cite{MR3422691}) shows that $r$ has a Lagrangian structure. In higher dimensions we conjecture that the map can be equipped with a shifted coisotropic structure in the sense of \cite{MR3848016, MR3848017}. 
\begin{conjecture}\label{conjecture}
The map $r$ can be naturally equipped with a coisotropic structure. 
\end{conjecture}
In order to avoid the analytic issues that arise in taking the holonomy, it may be preferable to replace $Loc_{G}(f^{-1}(S^{1}))$ with a moduli space of flat connections on the complement of $D$.

In recent work \cite{MR4322004, MR4436683}, Pantev and To\"{e}n studied the moduli spaces of local systems and flat connections on non-compact algebraic varieties. They constructed shifted Poisson structures and explained how to obtain the symplectic leaves by imposing suitable boundary conditions at infinity. Conjecture \ref{conjecture} may be viewed as providing another source of boundary conditions for the moduli spaces associated to $f^{-1}(\mathbb{C}^k \setminus D)$. We hope that it may also be used in conjunction with their results, for example by considering the map $r$ in the presence of additional boundary conditions at the boundary of the fibres of $f$.

One implication of the conjecture is that the moduli spaces $[\mathcal{W}(A)/Aut(S)]$ should admit $2(1-k)$-shifted Poisson structures. In Theorem \ref{tangentLiebialgebra} we provide evidence for the conjecture by constructing a $-2$-shifted Poisson structure on the formal neighbourhood of a special connection in the case of plane curves $x^{p} =  y^{q}$. Our construction is somewhat ad hoc, but it makes use of an invariant inner product on the Lie algebra $\mathfrak{g}$, as well as the intersection pairing on the cohomology of the curve $f^{-1}(1)$. We have also not checked that our shifted Poisson structure fits into the formalism developed by \cite{MR3653319}. We hope to address all these issues in future work. 

\vspace{.1in}

\noindent \textbf{Acknowledgements.} 
I would like to thank Elliot Cheung for pointing me to the paper \cite{MR3248985}. 

\section{Homogeneous free divisors and logarithmic flat connections} \label{hfdandlog}
Assume that the given $\mathbb{C}^{*}$ action on $\mathbb{C}^k$ has strictly positive weights. It is generated infinitesimally by an Euler vector field 
\[
E = \sum_{i = 1}^{k} n_{i} z_{i} \partial_{z_{i}},
\]
where $n_{i} \in \mathbb{Z}_{>0}$ are positive integers. This vector field defines a \emph{weight grading} on the holomorphic functions $\mathcal{O}_{\mathbb{C}^k}$ (and more generally tensor fields) on $\mathbb{C}^k$, such that the coordinate function $z_{i}$ has weight $n_{i}$. This grading will play an important role. Because of our assumption, each weight space is finite-dimensional over $\mathbb{C}$. We also assume that the function $f$ defining $D$ is homogeneous of weight $r$: $E(f) = rf$. 

The $\mathbb{C}^*$ action determines an action Lie algebroid $\mathbb{C} \ltimes \mathbb{C}^k$ which is generated by the Euler vector field. Because $E$ is a section of $T_{\mathbb{C}^k}(-\log D)$, there is an induced Lie algebroid morphism 
\[
i : \mathbb{C} \ltimes \mathbb{C}^k \to T_{\mathbb{C}^k}(-\log D), \qquad (\lambda, z) \mapsto \lambda E_{z}. 
\]
The logarithmic $1$-form $d \log f = \frac{df}{f}$ is a closed section of $\Omega^{1}_{\mathbb{C}^{k}}(\log D)$. Hence, it determines a Lie algebroid morphism 
\[
\pi : T_{\mathbb{C}^k}(-\log D) \to \mathbb{C}, \qquad V \mapsto \frac{1}{r f} V(f), 
\]
where $\mathbb{C}$ is considered as an abelian Lie algebra. The composition $p = \pi \circ i : \mathbb{C} \ltimes \mathbb{C}^k \to \mathbb{C}$ is given by projection to the first factor. This has a section 
\[
j : \mathbb{C} \to \mathbb{C} \ltimes \mathbb{C}^k, \qquad \lambda \mapsto (\lambda, 0),
\]
which is also a Lie algebroid morphism. Altogether, we have the following diagram of Lie algebroids: 
\[
\begin{tikzpicture}[scale=1.5]
\node (A) at (0,1) {$ \mathbb{C} \ltimes \mathbb{C}^k$};
\node (B) at (2,1) {$ \mathbb{C}$};
\node (E) at (1,0) {$T_{\mathbb{C}^k}(-\log D) $};

\path[->,>=angle 90]
(A) edge node[below]{$i$} (E)
(E) edge node[below]{$\pi$} (B);

\path[->] (0.4,1.05) edge node[above]{$p$} (1.8,1.05);
\path[->] (1.8,0.95) edge node[below]{$j$} (0.4,0.95);
\end{tikzpicture}
\]
Each Lie algebroid determines a differential graded Lie algebra, whose Maurer-Cartan locus consists of flat algebroid connections. Furthermore, each morphism of Lie algebroids determines a pullback morphism between dgla's, and as a result, a pullback morphism between categories of representations, or more generally, derived moduli stacks of flat connections. This gives rise to the following diagram of (infinite-dimensional) derived stacks: 
\[
\begin{tikzpicture}[scale=1.5]
\node (A) at (0,1) {$[(\mathcal{O}_{\mathbb{C}^{k}} \otimes \mathfrak{g}) / \mathfrak{G} ]$};
\node (B) at (2,1) {$[\mathfrak{g}/G]$};
\node (E) at (1,0) {$[\mathcal{M}_{D, \mathfrak{g}} /\mathfrak{G} ]$};

\path[->,>=angle 90]
(E) edge node[below]{$i^*$} (A)
(B) edge node[below]{$\pi^*$} (E);

\path[->] (1.4,1.05) edge node[above]{$p^*$}  (0.8,1.05);
\path[->] (0.8,0.95) edge node[below]{$j^*$} (1.4,0.95);
\end{tikzpicture}
\]

In this diagram, $[\mathfrak{g}/G]$ is the moduli stack of $\mathfrak{g}$-representations of $\mathbb{C}$. It is the stack quotient corresponding to the adjoint action of $G$ on its Lie algebra. $[(\mathcal{O}_{\mathbb{C}^{k}} \otimes \mathfrak{g}) / \mathfrak{G} ]$ is the moduli stack of $\mathfrak{g}$-representations of $\mathbb{C} \ltimes \mathbb{C}^r$. In both cases the derived structure is trivial because the Lie algebroids have rank $1$. 

Now fix an element $A \in \mathfrak{g}$, let $O_{A} \subset \mathfrak{g}$ be its adjoint orbit, and let $G_{A} \subseteq G$ be its centralizer subgroup. This determines a substack $[O_{A}/G] \subset [\mathfrak{g}/G]$ which is Morita equivalent to $BG_{A}$. The preimage $[\mathcal{M}_{D, \mathfrak{g}}(A)/\mathfrak{G}] := (j^* i^*)^{-1}(BG_{A})$ is the derived stack of logarithmic flat connections $\omega$ whose `residue' $j^* i^*(\omega)$ lies in $O_{A}$. More precisely, the base of the derived manifold $\mathcal{M}_{D, \mathfrak{g}}(A)$ is given by 
\[
M(A) = \{ \omega \in \Omega^{1}_{\mathbb{C}^{k}}(\log D) \otimes \mathfrak{g} \ | \ j^* i^* (\omega) \in O_{A} \},
\]
with the bundle of curved dgla's restricted from $\mathcal{M}_{D, \mathfrak{g}}$. The action of $\mathfrak{G}$ preserves $M(A)$. 

\section{Finite dimensional model} \label{finitedimmodel}
Let $A = S + N_{0}$ be the Jordan decomposition of $A$, where $S$ is semisimple, $N_{0}$ is nilpotent, and $[S,N_{0}] = 0$. In this section we will construct a finite dimensional model for $[\mathcal{M}_{D, \mathfrak{g}}(A)/\mathfrak{G}] $.

\subsection*{The dgla $L_{D, \mathfrak{g}}$}

We start by analysing the structure of the dgla $L_{D, \mathfrak{g}}$. Being constructed from the cdga $\Omega^{\bullet}_{\mathbb{C}^{k}}(\log D)$ and the Lie algebra $\mathfrak{g}$, $L_{D, \mathfrak{g}}$ inherits their derivations. The basic ones are as follows: 
\begin{itemize}
\item the Lie algebroid differential $d$, which has degree $+1$ and squares to $0$,
\item the interior multiplication with the Euler vector field $\iota_{E}$, which has degree $-1$ and squares to $0$,
\item the adjoint action of $S$, $ad_{S}$, which has degree $0$. 
\end{itemize}
By taking commutator brackets we arrive at further derivations, such as $L_{E} = [\iota_{E}, d]$, the Lie derivative with respect to $E$, which is a derivation of degree $0$. We can also wedge any derivation by a differential form to obtain a new derivation. Let $\alpha_{0} = \frac{1}{r}d\log f$, which is a closed logarithmic $1$-form. Then $\alpha_{0} ad_{S}$ is a degree $+1$ derivation which squares to $0$. Among the $5$ derivations just described, almost all of them commute. The only two non-vanishing commutator brackets are the following:
\[
[\iota_{E}, d] = L_{E}, \qquad [\iota_{E}, \alpha_{0} ad_{S}] = ad_{S}. 
\]
The second bracket follows as a consequence of the identity $\iota_{E}(\alpha_{0}) = 1$. We are primarily interested in studying the dgla structure arising from 
\[
\delta_{S} = d + \alpha_{0} ad_{S},
\]
which is a degree $+1$ derivation that squares to $0$. We are also interested in the following degree $0$ derivation
\[
L_{S} := [\iota_{E}, \delta_{S}] = L_{E} + ad_{S}. 
\]
This operator is diagonalisable in the sense that any element $\beta \in L_{D, \mathfrak{g}}$ has a Taylor series expansion 
\[
\beta = \sum_{u} \beta_{u}
\]
where each term satisfies $L_{S}(\beta_{u}) = u \beta_{u}$. Indeed, the operator $ad_{S}$ is diagonalizable on $\mathfrak{g}$ with finitely many eigenvalues since $S$ is semisimple. The eigenspaces of $L_{E}$ are the weight spaces. We noted earlier that the weight degrees of holomorphic functions are strictly positive integers, and that each weight space is finite dimensional. As an operator on $\Omega^{\bullet}_{\mathbb{C}^{k}}(\log D)$, the eigenvalues of $L_{E}$ may not be positive, but they are integers which are bounded below. Hence, the eigenvalues of $L_{S}$ have the form $u_{i} + \mathbb{Z}_{\geq 0}$, for finitely many complex numbers $u_{i}$.  

Let $L_{D, \mathfrak{g}, u}$ denote the $u$-eigenspace, and note that it is finite-dimensional. Because $L_{S}$ is a derivation, the Lie bracket respects this decomposition: 
\[
[ -, - ] : L_{D, \mathfrak{g}, u} \times L_{D, \mathfrak{g}, v} \to L_{D, \mathfrak{g}, u + v}. 
\]
In particular, $L_{D, \mathfrak{g}, 0}$ is a finite-dimensional Lie subalgebra. The derivations $\delta_{S}$ and $\iota_{E}$ commute with $L_{S}$, and hence preserve its eigenspaces. In particular, they restrict to $L_{D, \mathfrak{g}, 0}$. 

Now introduce the degree $0$ derivation $P = \alpha_{0} \iota_{E}$. This derivation satisfies $P^2 = P$, and hence induces a decomposition $L_{D, \mathfrak{g}} = \ker(P) \oplus \im(P)$. Let $U = \ker(P)$ and let $I = \im(P)$. With respect to the bracket, $U$ is a subalgebra and $I$ is an abelian ideal. 
\begin{lemma} \label{decomposition}
The derivation $\iota_{E}$ vanishes on $U$. For every degree $i$ it defines an isomorphism 
\[
\iota_{E} : I^{i} \to U^{i-1},
\]
with inverse given by multiplication by $\alpha_{0}$. Therefore, as a graded Lie algebra, $L_{D, \mathfrak{g}}$ is isomorphic to $U \ltimes U[-1]$, where $U$ acts on $U[-1]$ via the adjoint action. 
\end{lemma}
\begin{proof}
It is clear that $\ker(\iota_{E}) \subseteq U = \ker(P)$. For the opposite inclusion, suppose that $P(x) = 0$. Then $\iota_{E}(x)$ is in the kernel of multiplication by $\alpha_{0}$. Since $\alpha_{0}$ is a non-vanishing algebroid $1$-form, $\iota_{E}(x)$ must be of the form $\alpha_{0} \wedge y$. But then 
\[
0 = \iota^{2}_{E}(x) = \iota_{E}(\alpha_{0} \wedge y) = y - \alpha_{0} \wedge \iota_{E}(y),
\]
which implies that $\iota_{E}(x) = 0$, as required. The image of $\iota_{E}$ is contained in $U$ since $\iota_{E}^2 = 0$. To see surjectivity, we can explicitely construct the inverse as mulitplication by $\alpha_{0}$. Given $x \in U$, check that $\alpha_{0} \wedge x = P(\alpha_{0} \wedge x) \in I$. Hence $\alpha_{0}\wedge : U^{i-1} \to I^{i}$. Then for $x \in U$, we have $\iota_{E}(\alpha_{0} \wedge x) = x$, and for $P(y) \in I$ we have $\alpha_{0} \wedge \iota_{E} P(y) = P^2(y) = P(y)$. 

Now define the isomorphism $\Xi : L_{D, \mathfrak{g}} \to U \ltimes U[-1]$ by the following formula in degree $i$: 
\[
U^{i} \oplus I^{i} \to U^{i} \oplus U^{i-1}, \qquad (x,y) \mapsto (x, (-1)^{i}\iota_{E}(y)). 
\]
This preserves Lie brackets. 
\end{proof}

The commutator $[P, L_{S}] = 0$. Therefore, the two operators can be simultaneously diagonalized. In particular, we have the decomposition $L_{D, \mathfrak{g}, 0} = U_{0} \oplus I_{0}$. The results of the previous lemma remain true for this subalgebra. Next, we have $[P, \delta_{S}] = \alpha_{0} L_{S}$. If we re-write this as the following identity 
\[
P \delta_{S} = \alpha_{0} L_{S} + \delta_{S} P
\]
then we can deduce that $\delta_{S}$ preserves $I$. Indeed, applying this identity to an element of the form $x = P(y)$, we obtain 
\[
P \delta_{S}(x) = \alpha_{0} L_{S} P(y) + \delta_{S} P^2(y) = \alpha_{0} P L_{S}(y) + \delta_{S}P(y) = \delta_{S}(x). 
\]
On the other hand, the differential $\delta_{S}$ does not preserve $U$. But by applying the identity to an element $x \in U$, we compute that the `off-diagonal' term is given by $P \delta_{S}(x) = \alpha_{0} L_{S}(x)$. This term vanishes when we restrict to the subalgebra $L_{D, \mathfrak{g}, 0}$. Hence, we obtain the following corollary. 

\begin{corollary} \label{shiftedadjoint}
The subalgebra $U_{0}$ is preserved by $\delta_{S}$, and there is an isomorphism of dgla's 
\[
(L_{D, \mathfrak{g}, 0}, \delta_{S}) \cong (U_{0}, \delta_{S}) \ltimes (U_{0}, \delta_{S})[-1]. 
\]
\end{corollary}
\begin{proof}
On the subspace $L_{D, \mathfrak{g}, 0}$ we have $[\iota_{E}, \delta_{S}] = 0$. This implies that the morphism $\Xi$ from Lemma \ref{decomposition} is a chain map. 
\end{proof}

\subsection*{The gauge group $Aut(S)$}
Viewing $S \in \mathfrak{g}$ as a representation of $\mathbb{C}$, we can pull it back to obtain a representation $p^*S$ of $\mathbb{C} \ltimes \mathbb{C}^{k}$. Let $Aut(S)$ be the subgroup of the gauge group $\mathfrak{G}$ consisting of gauge transformations which preserve $p^*S$: 
\[
Aut(S) = \{ g \in \mathfrak{G} \ | \ g \ast p^* S = p^* S \}. 
\]
It is a finite-dimensional algebraic group whose Lie algebra is $L^{0}_{D, \mathfrak{g}, 0}$. We recall the description of its Levi decomposition which was given in \cite{bischoff2022normal}. The automorphism group of $j^* p^* S = S$ is $G_{S}$, the centralizer subgroup of $S$ in $G$, which is reductive. The pullback functor $j^*$ defines a homomorphism 
\[
j^{*}: Aut(S) \to G_{S}, \qquad g \mapsto g(0),
\]
and the pullback functor $p^*$ defines a splitting. The kernel of $j^*$, denoted $Aut_{0}(S)$, is the unipotent radical. Hence the isomorphism $Aut(S) \cong Aut_{0}(S) \rtimes G_{S}$ provides the Levi decomposition. 

Define the following \emph{gauge action} of $Aut(S)$ on $L^1_{D, \mathfrak{g}, 0}$: 
\[
g \ast x = g x g^{-1} - \delta_{S}(g)  g^{-1},
\]
where $\delta_{S}(g) g^{-1} = dg g^{-1} + \alpha_{0}(S - gSg^{-1}).$ 
\begin{lemma} \label{gaugeaction}
The gauge action of $Aut(S)$ is well-defined. In terms of the decomposition $L^1_{D, \mathfrak{g}, 0} = U^{1}_{0} \oplus I^{1}_{0}$ it is given by 
\[
g \ast (x, y) = (g x g^{-1} - \delta_{S}(g)  g^{-1}, g y g^{-1}),
\]
where $x \in U^{1}_{0}$ and $y \in I^{1}_{0}$. Furthermore, $Aut(S)$ acts on $L^{\bullet \geq 2}_{D, \mathfrak{g}, 0} $ by conjugation, preserving the decomposition $U_{0} \oplus I_{0}$ and the Lie bracket. 
\end{lemma}
\begin{proof}
A computation shows that $L_{S}(g \ast x) = g(L_{S}x)g^{-1}$ for $x \in L^1_{D, \mathfrak{g}}$, showing that $L^1_{D, \mathfrak{g}, 0}$ is preserved. Similarly, $L_{S}(g x g^{-1}) = g(L_{S}x)g^{-1}$ for $x \in L^j_{D, \mathfrak{g}}$, showing that the conjugation action preserves $L^{\bullet \geq 2}_{D, \mathfrak{g}, 0} $. Next, for $x \in L_{D, \mathfrak{g}, 0}$ we have $P(gxg^{-1}) = g P(x) g^{-1}$, implying that the conjugation also preserves $U_{0}$ and $I_{0}$. Finally, 
\[
P(\delta_{S}(g) g^{-1}) = \alpha_{0}(L_{E}(g) g^{-1} + S - g S g^{-1}),
\]
which vanishes for $g \in Aut(S)$. Hence $\delta_{S}(g) g^{-1} \in U^{1}_{0}$. 
\end{proof}

\subsection*{The finite-dimensional derived stack}
Given the finite dimensional dgla $L_{D, \mathfrak{g}, 0}$ we obtain a derived manifold $\mathcal{W}_{S}$. The base manifold is the vector space $W_{S} = L^{1}_{D, \mathfrak{g}, 0}$, the bundle of curved dgla's is the trivial bundle $W_{S} \times L^{\bullet \geq 2}_{D, \mathfrak{g}, 0} $, the curvature section is given by the standard formula $F_{S}(w) = \delta_{S}(w) + \frac{1}{2}[w, w]$, and the twisted differential $\delta$ is given by 
\[
\delta_{S,w} = \delta_{S} + [w, - ],
\]
for $w \in W_{S}$. Furthermore, Lemma \ref{gaugeaction} gives an equivariant action of $Aut(S)$ on $W_{S} \times L^{\bullet \geq 2}_{D, \mathfrak{g}, 0} $, preserving the bracket. It is also straightforward to check that this action preserves $F_{S}$ and $\delta$. Hence, we obtain a derived stack $[\mathcal{W}_{S}/Aut(S)]$. 

$U_{0}$ is a sub-dgla of $L_{D, \mathfrak{g}, 0}$, which is preserved by the action of $Aut(S)$. Hence, it gives rise to a derived substack $[\mathcal{U}_{S}/Aut(S)]$ of $[\mathcal{W}_{S}/Aut(S)]$. Furthermore, since $I_{0}$ is an ideal of $L_{D, \mathfrak{g}, 0}$, we also get a projection morphism $ [\mathcal{W}_{S}/Aut(S)] \to [\mathcal{U}_{S}/Aut(S)].$

\begin{proposition}
The derived stack $[\mathcal{W}_{S}/Aut(S)]$ is isomorphic to the shifted tangent bundle $T[-1] [\mathcal{U}_{S}/Aut(S)]$. 
\end{proposition}
\begin{proof}
This follows from Corollary \ref{shiftedadjoint} and Lemma \ref{gaugeaction}.
\end{proof}

We are actually interested in a substack of $[\mathcal{W}_{S}/Aut(S)]$ which is determined by the element $A = S + N_{0}$. Recall that the image of $Aut(S)$ under $j^{*}$ is $G_{S}$, the centralizer of $S$. This implies that for any element $\omega \in W_{S}$, the image $j^{*} i^{*}(\omega) \in \mathfrak{g}_{S} = Lie(G_{S})$. We will require that this element be contained in $G_{S} \ast N_{0}$, the adjoint orbit of $N_{0}$ in $\mathfrak{g}_{S}$. Namely, define 
\[
W(A) = \{ \omega \in W_{S} \ | \ j^{*} i^{*}(\omega) \in G_{S} \ast N_{0} \}. 
\]
Let $\mathcal{W}(A)$ be the derived manifold obtained by pulling back the bundle of curved dgla's from $W_{S}$ to $W(A)$. The action of $Aut(S)$ restricts to an action on this sub-manifold. Hence, we obtain a derived stack $[\mathcal{W}(A)/Aut(S)]$. 

\begin{theorem} \label{Maintheorem}
$[\mathcal{W}(A)/Aut(S)]$ is equivalent to $[\mathcal{M}_{D, \mathfrak{g}}(A)/\mathfrak{G}] $, the derived stack of logarithmic flat connections whose residue lies in the adjoint orbit $O_{A}$ of $A$. 
\end{theorem}

\section{Proof of Theorem \ref{Maintheorem}} 

In this section we will give the proof of the equivalence between $[\mathcal{W}(A)/Aut(S)]$ and $[\mathcal{M}_{D, \mathfrak{g}}(A)/\mathfrak{G}] $. There is a natural morphism 
\[
q: [\mathcal{W}(A)/Aut(S)] \to [\mathcal{M}_{D, \mathfrak{g}}(A)/\mathfrak{G}],
\]
which we describe as follows: 
\begin{enumerate}
\item The map on the base manifolds is given by the following formula 
\[
q: W(A) \to M(A), \qquad \omega \mapsto \alpha_{0} S + \omega.
\]
\item The map on bundles of curved dgla is given by the inclusion $L^{\bullet \geq 2}_{D, \mathfrak{g}, 0} \to L^{\bullet \geq 2}_{D, \mathfrak{g}}$. 
\item The group $Aut(S)$ includes into $\mathfrak{G}$ as a subgroup, and the map $q$ is equivariant. 
\end{enumerate}
In order to show that $q$ is an equivalence, we must show two things. First, there is an underlying functor between the classical groupoids: 
\[
\pi_{0}(q) : Aut(S) \ltimes MC(\mathcal{W}(A)) \to \mathfrak{G} \ltimes MC(\mathcal{M}_{D, \mathfrak{g}}(A)). 
\]
We need to show that this is an equivalence of categories. This is implied by \cite[Theorem 5.5]{bischoff2022normal} and the following lemma.

\begin{lemma} \label{AutomaticNilpotency}
Let $\omega \in W(A)$. Then $\iota_{E}(\omega)$ is nilpotent. 
\end{lemma}
\begin{proof}
For $\omega \in L_{D, \mathfrak{g}, 0}^{1}$, we have $\iota_{E}(\omega) \in U^{0}_{0} = Lie(Aut(S))$. If $\omega \in W(A)$ we have in addition that $j^{*} \iota_{E}(\omega) \in G_{S} \ast N_{0}$, and so is nilpotent. Let $\iota_{E}(\omega) = B_s + B_{n}$ be the Jordan decomposition, where $B_{s}$ is semisimple and $B_{n}$ is nilpotent. Then $j^{*}(B_{s}) = 0$, so that $B_{s} \in Lie(Aut_{0}(S))$. But since $Aut_{0}(S)$ is unipotent, this implies that $B_{s} = 0$, and hence $\iota_{E}(\omega)$ is nilpotent. 
\end{proof}

Second, a derived stack has a tangent complex at every point of its MC locus, and the map $q$ induces a chain map between the tangent complexes:
\[
dq_{w} : \mathbb{T}_{w} [\mathcal{W}(A)/Aut(S)] \to \mathbb{T}_{q(w)} [\mathcal{M}_{D, \mathfrak{g}}(A)/\mathfrak{G}].
\]
We need to show that this is a quasi-isomorphism at each point of the MC locus. We will do this by first constructing an explicit homotopy at the special point $q(0)$ (which is generally not in our space), and then apply the homological perturbation lemma to obtain the quasi-isomorphism at all points. 

\subsection*{The homotopy}
Let $a : L_{D, \mathfrak{g}, 0} \to L_{D, \mathfrak{g}}$ be the inclusion and let $b : L_{D, \mathfrak{g}} \to L_{D, \mathfrak{g}, 0}$ be the projection to the degree $0$ component. Both $a$ and $b$ are chain maps with respect to $\delta_{S}$, but in general only $a$ preserves the Lie bracket. Furthermore, $b \circ a = id_{L_{D, \mathfrak{g}, 0}}$. 

Recall that a given element $\beta \in L_{D, \mathfrak{g}}$ has a Taylor expansion in the eigenvalues of $L_{S}$: 
\[
\beta = \sum_{u} \beta_{u}, 
\]
where each term satisfies $L_{S}(\beta_{u}) = u \beta_{u}$. As we saw, the eigenvalues have the form $u_{i} + \mathbb{Z}_{\geq 0}$ for finitely many complex numbers $u_{i}$. For this reason the series 
\[
\beta' = \sum_{u \neq 0} \frac{1}{u} \beta_{u}
\]
converges to a well-defined element of $L_{D, \mathfrak{g}}$. We use this to define the following degree $-1$ operator
\[
h : L^{i}_{D, \mathfrak{g}} \to L^{i-1}_{D, \mathfrak{g}}, \qquad \sum_{u} \beta_{u} \mapsto \iota_{E}(\sum_{u \neq 0} \frac{1}{u} \beta_{u}). 
\]

The following lemma results from straightforward computation. It has the upshot that $a$ defines a quasi-isomorphism of dgla's from $(L_{D, \mathfrak{g}, 0}, \delta_{S}) $ to $(L_{D, \mathfrak{g}}, \delta_{S}) $.

\begin{lemma} \label{hmtpy}
The operator $h$ defines a homotopy between $ab$ and $id_{L_{D, \mathfrak{g}}}$. In other words, it satisfies 
\[
[\delta_{S}, h] = id_{L_{D, \mathfrak{g}}} - ab. 
\]
Furthermore, it satisfies the `side conditions' $h \circ a = 0$, $b \circ h = 0$ and $h^2 = 0$. Finally, it vanishes on $U$ and sends $I^{i}$ to $U^{i-1}$. 
\end{lemma}

\subsection*{The perturbation}
We will now perturb the differential $\delta_{S}$ and show that $a$ continues to define a quasi-isomorphism. This is achieved by using the perturbation lemma \cite{crainic2004perturbation}. 

Let $w \in W(A)$ satisfy the Maurer-Cartan equation $\delta_{S}(w) + \frac{1}{2}[w, w] = 0$ and consider the perturbed differential $\delta_{S,w} = \delta_{S} + [w, -]$. This is a differential on $L_{D, \mathfrak{g}}$, and we want an induced perturbation of the homotopy data $(a, b, h, \delta_{S})$ of the previous section. 

\begin{lemma}
The endomorphism $ad_{w} \circ h$ of $L_{D, \mathfrak{g}}$ is nilpotent. 
\end{lemma}
\begin{proof}
The element $w$ can be decomposed as $w = \gamma + \alpha_{0} N$, where $\gamma \in U^{1}_{0}$ and $N \in U^{0}_{0}$. By Lemma \ref{AutomaticNilpotency}, $N$ is nilpotent. Recall from Lemma \ref{hmtpy} that $h$ vanishes on $U$ and its image is contained in $U$. Furthermore, since $U$ is a subalgebra of $L_{D, \mathfrak{g}}$, it is preserved by $ad_{\gamma}$. As a result $h \circ ad_{\gamma} \circ h = 0$. Hence, it suffices to show that the operator $\alpha_{0} ad_{N} \circ h$ is nilpotent. 

Now note that $ad_{N}$ and multiplication by $\alpha_{0}$ commute. Since $N \in U^{0}_{0}$, $ad_{N}$ also commutes with $h$. This implies that 
\[
(h \circ \alpha_{0} ad_{N} )^{k} = (ad_{N})^{k} \circ \tilde{h}^{k},
\]
where $\tilde{h}$ is the operator $\tilde{h}(\beta) = h(\alpha_{0} \wedge \beta)$. But this will vanish for large enough $k$ since $N$ is nilpotent. 
\end{proof}

The upshot of this lemma is that we can now define the following perturbed maps: 
\begin{align*}
h' &= \sum_{p = 0}^{\infty} (-h  ad_{w})^{p} h, \\ 
\delta' &= \delta_{S} + \sum_{p = 0}^{\infty} b (-ad_{w}  h)^{p} ad_{w} a, \\ 
a' &= \sum_{p = 0}^{\infty} (-h  ad_{w})^{p} a, \\ 
b' &= \sum_{p=0}^{\infty} b (-ad_{w}  h)^{p}. 
\end{align*}

The perturbation lemma says that $\delta'$ defines a differential on $L_{D, \mathfrak{g}, 0}$, that $a'$ and $b'$ define chain maps between $(L_{D, \mathfrak{g}, 0}, \delta')$ and $(L_{D, \mathfrak{g}}, \delta_{S,w})$, and that the following equations are satisfied: 
\[
b' \circ a' = id_{L_{D, \mathfrak{g}, 0}}, \ \  [\delta_{S,w}, h'] = id_{L_{D, \mathfrak{g}}} - a' \circ b', \ \ h' \circ a' = 0, \ \ b' \circ h' = 0, \ \ h' \circ h' = 0. 
\]

The following lemma identifies the perturbations.  
\begin{lemma} \label{perthmtpy}
The perturbations are given by 
\[
a' = a, \ \ b' = b, \ \ \delta' = \delta_{S,w}.
\]
In particular, the inclusion $a : (L_{D, \mathfrak{g}, 0}, \delta_{S,w}) \to (L_{D, \mathfrak{g}}, \delta_{S,w})$ is a quasi-isomorphism. Furthermore, $h'$ vanishes on $U$ and sends $I$ to $U$. 
\end{lemma}
\begin{proof}
The element $w \in L_{D, \mathfrak{g}, 0}^{1}$ and so $ad_{w}$ restricts to $L_{D, \mathfrak{g}, 0}$ and commutes with both $a$ and $b$. As a result of this and the side conditions of Lemma \ref{hmtpy}, we have that $h ad_{w}a = 0$ and $b ad_{w} h = 0$. Plugging this into the definitions of the deformed maps gives 
\begin{align*}
a' &= a - \sum_{p \geq 0}(-h ad_{w})^p (h ad_{w} a) = a, \\
\delta' &= \delta_{S} + b ad_{w} a - \sum_{p \geq 0} b (-ad_{w} h)^{p-1} ad_{w} (h ad_{w} a) = \delta_{S,w}, \\
b' &= b - \sum_{p \geq 0} (b ad_{w} h) (-ad_{w} h)^p = b. 
\end{align*}
The statement about $h'$ follows from Lemma \ref{hmtpy} and the fact that each term in the definition of $h'$ starts and ends with $h$. 
\end{proof}

\subsection*{The quasi-isomorphism of tangent complexes}
Consider a point $w \in MC(\mathcal{W}(A))$. It has the form $w = \gamma + \alpha_{0} N$, where $\gamma \in U^{1}_{0}$, $N \in U^{0}_{0}$, and $j^* i^*(w) = N(0) \in G_{S} \ast N_{0}$. It has corresponding point $q(w) \in MC(\mathcal{M}_{D, \mathfrak{g}}(A))$. In this section we will describe the morphism of tangent complexes $dq_{w} : \mathbb{T}_{w} [\mathcal{W}(A)/Aut(S)] \to \mathbb{T}_{q(w)} [\mathcal{M}_{D, \mathfrak{g}}(A)/\mathfrak{G}]$ and show that it is a quasi-isomorphism. 

We start by describing the tangent complexes. First, the tangent complex of $[\mathcal{M}_{D, \mathfrak{g}}(A)/\mathfrak{G}]$ is given as follows: 
\[
 \mathbb{T}_{q(w)} [\mathcal{M}_{D, \mathfrak{g}}(A)/\mathfrak{G}] = L^{0}_{D, \mathfrak{g}} \to T_{q(w)}M(A) \to L^{2}_{D, \mathfrak{g}} \to L^{3}_{D, \mathfrak{g}} \to ... 
\]
Note that the first term is $L^{0}_{D, \mathfrak{g}}  = Lie(\mathfrak{G})$, and the second term is the subspace 
\[
T_{q(w)}M(A) = \{ v \in L^{1}_{D, \mathfrak{g}} \ | \ j^* i^*(v) \in T_{(S + N(0))}O_{A} \}, 
\]
where we use the fact that $j^* i^* (q(w)) = S + N(0)$. The first map is the derivative of the gauge action, and a computation shows that it is equal to $-\delta_{S,w}$. The minus sign is due to the fact that we are making the gauge group act on the left. For simplicity we will replace this by $\delta_{S,w}$, since it does not affect the cohomology. The second map is the derivative of the curvature $dF$, and a calculation shows that it is given by $\delta_{S,w}$. Finally, all higher maps are given by $\delta_{q(w)} = \delta_{S, w}$. Therefore, the tangent complex is a subcomplex of $(L_{D, \mathfrak{g}}, \delta_{S,w})$. 

The tangent complex of $[\mathcal{W}(A)/Aut(S)]$ has a similar descriptions. It is given by 
\[
\mathbb{T}_{w} [\mathcal{W}(A)/Aut(S)] = L^{0}_{D, \mathfrak{g}, 0} \to T_{w}W(A) \to L^{2}_{D, \mathfrak{g}, 0} \to L^{2}_{D, \mathfrak{g}, 0} \to ...
\]
As above, $L^{0}_{D, \mathfrak{g}, 0} = Lie(Aut(S))$ and the second term is the subspace 
\[
 T_{w}W(A) = \{ v \in  L^{1}_{D, \mathfrak{g}, 0} \ | \ j^* i^*(v) \in T_{N(0)}(G_{S} \ast N_{0}) \}. 
\]
Again all maps are given by $\delta_{S,w}$ (the first map has a minus sign, which we remove for simplicity). Hence, the tangent complex is a subcomplex of $(L_{D, \mathfrak{g}, 0}, \delta_{S,w})$. 

The map $dq_{w}$ is easily seen to coincide with $a$. Therefore, in order to prove that $dq_{w}$ is a quasi-isomorphism, it suffices to show that the homotopy data $(a, b, h', \delta_{S,w})$ restricts to the tangent complexes. 

\begin{lemma}
The maps $(a, b, h', \delta_{S,w})$ restrict to $\mathbb{T}_{q(w)} [\mathcal{M}_{D, \mathfrak{g}}(A)/\mathfrak{G}]$ and $\mathbb{T}_{w} [\mathcal{W}(A)/Aut(S)]$. Therefore, $dq_{w}$ defines a quasi-isomorphism. 
\end{lemma}
\begin{proof}
Since the complexes are modified in degree $1$, it suffices to restrict our attention to degrees $0, 1, 2$. The above description of the tangent complexes and $dq_{w}$ immediately implies that $a$ and $\delta_{S, w}$ restrict. To check that $h'$ restricts, we only need to show that $h'(L^{2}_{D, \mathfrak{g}})$ is contained in $T_{q(w)}M(A)$. But this follows because, by Lemma \ref{perthmtpy}, the image of $h'$ is contained in $U$. 

For the map $b$, consider a point $\beta \in T_{q(w)}M(A)$. This can be expanded as $\beta = \sum_{u} \beta_{u}$, where each term satisfies $L_{S}(\beta_{u}) = u \beta_{u}$. By definition $b(\beta) = \beta_{0}$. Hence, we need to check that if $j^* i^* (\beta) \in T_{(S + N(0))}O_{A}$, then $j^* i^* (\beta_{0}) \in T_{N(0)}(G_{S} \ast N_{0})$. These tangent spaces have the following descriptions
\[
T_{(S + N(0))}O_{A} = Im( ad_{S + N(0)} : \mathfrak{g} \to \mathfrak{g}), \qquad T_{N(0)}(G_{S} \ast N_{0}) = Im( ad_{N(0)} : \mathfrak{g}_{S} \to \mathfrak{g}_{S}).
\]
Now using the eigenvector expansion, we have 
\[
j^* i^* (\beta) = \sum_{u} j^{*} i^{*}(\beta_{u}) = ad_{S + N(0)}(Z), 
\]
for some $Z \in \mathfrak{g}$. One can check that each term in the summand satisfies $ad_{S}( j^{*} i^{*}(\beta_{u}) ) = u  j^{*} i^{*}(\beta_{u}) $. Since $ad_{S} : \mathfrak{g} \to \mathfrak{g}$ is diagonalizable, we can decompose $Z$ into eigenvectors as well: $Z = \sum_{u} Z_{u}$. And since $ad_{S + N(0)}$ commutes with $ad_{S}$, it preserves the eigenspaces. Hence, we can match up the eigenvectors to get 
\[
j^* i^* \beta_{0} = ad_{S + N(0)}(Z_{0}) = ad_{N(0)}(Z_{0}), 
\]
where $Z_{0} \in \mathfrak{g}_{S}$. 
\end{proof}

\section{Plane curves $x^{p} - y^{q}$} \label{planecurvesection}
In this section we give a detailed study of the case of plane curves. Consider \[ f = x^p - y^q : \mathbb{C}^2 \to \mathbb{C}, \] where $p$ and $q$ are relatively prime positive integers satisfying $p < q$. This function is weighted homogeneous of degree $qp$ for the Euler vector field $E = qx \partial_{x} + py \partial_{y}$, which defines the weight grading on coordinates $|x| = q$ and $|y| = p$. The logarithmic tangent bundle is generated by the vector fields $E$ and $V = qy^{q-1} \partial_{x} + px^{p-1}\partial_{y}$, which satisfy $[E,V] = (qp - p - q)V$. Let $w_{0} = qp - p - q$. The logarithmic $1$-form $\alpha_{0} = \frac{1}{qp}d \log f$ pairs with $V$ to give $0$. Therefore, it can be completed to a dual basis $\alpha_{0}, \beta$ of the logarithmic cotangent bundle. The form $\alpha_{0}$ is closed, and $\beta$ satisfies $d \beta = (p + q - qp) \alpha_{0} \wedge \beta$. 

\subsection*{Cohomology of $V$}
Let $\mathcal{O}_{w}$ denote the subspace of polynomial functions with weight $w$ with respect to $E$. Note that any integer $w \in \mathbb{Z}$ has a unique decomposition $w = aq + bp + cqp$, where $a, b, c \in \mathbb{Z}$, $0 \leq a < p$ and $0 \leq b < q$. This decomposition provides a useful way of indexing the weights because of the following lemma. 

\begin{lemma} \label{weightspacedecompV}
Let $w = aq + bp + cqp$, with the above restrictions on $a, b, c$. The dimension of $\mathcal{O}_{w}$ is $\max(c,-1) + 1$, and a basis is given by 
\[
x^{a + cp}y^{b}, x^{a + (c-1)p}y^{b+q}, ..., x^{a}y^{b + cq}. 
\]
\end{lemma}
The vector field $V$ has weight $w_{0}$, and hence it defines a map \[
V : \mathcal{O}_{w} \to \mathcal{O}_{w + w_{0}}.
\]
\begin{lemma} \label{kernelofV}
The kernel of $V$ is $\mathbb{C}[f]$. 
\end{lemma}
\begin{proof}
A calculation shows that $V(f^{c}) = 0$. Conversely, let $g \in \ker(V)$. Because $V$ is homogeneous, it suffices to consider the case where $g$ is homogeneous of weight $w > 0$. The equation $V(g) = 0$ implies that $\partial_{x}g = px^{p-1}h$ and $\partial_{y}g = -qy^{q-1}h$, for a common polynomial $h$. Therefore, 
\[
wg = E(g) = qx \partial_{x}g + py \partial_{y} g = qp(x^p - y^q) h,
\]
so that $g = \frac{qp}{w} f h$. Hence $h$ is a function of weight $w - qp$ and it lies in the kernel of $V$. The result now follows by induction on the weight. 
\end{proof}

The Jacobian ideal of $f$ is generated by $x^{p-1}$ and $y^{q-1}$. Let $C = \mathbb{C}[x,y]/(x^{p-1}, y^{q-1})$, considered as a $\mathbb{C}$-vector space. It has a natural basis of monomials $x^{a}y^{b}$, where $0 \leq a < p-1$ and $0 \leq b < q-1$. Using this basis, $C$ is naturally graded by weight, and there is a weight preserving injective linear map $C \to \mathbb{C}[x,y]$. Consider the graded polynomial ring $\mathbb{C}[f]$, where $f$ has degree $pq$. Then $\mathbb{C}[x,y]$ is a graded $\mathbb{C}[f]$-module and there is a morphism of graded $\mathbb{C}[f]$-modules 
\[
\mathbb{C}[f] \otimes_{\mathbb{C}} C \to \mathbb{C}[x,y]. 
\]
The action of $V$ on $\mathbb{C}[x,y]$ is $\mathbb{C}[f]$-linear, so that the cokernel $\spc{coker}(V)$ is also a $\mathbb{C}[f]$-module. Post-composing with the quotient projection, we obtain the morphism 
\[
\mathbb{C}[f] \otimes_{\mathbb{C}} C \to \spc{coker}(V). 
\]
\begin{lemma} \label{cokernelofV}
The morphism $\mathbb{C}[f] \otimes_{\mathbb{C}} C \to \spc{coker}(V)$ is an isomorphism of graded $\mathbb{C}[f]$-modules. 
\end{lemma}
\begin{proof}
Since $V$ is homogeneous it suffices to consider a single weight at a time: we consider the cokernel of the map $V : \mathcal{O}_{w - w_{0}} \to \mathcal{O}_{w}$. Let $w = aq + bp + cqp$, where $0 \leq a < p$, $0 \leq b < q$ and $c \geq 0$, so that $\mathcal{O}_{w}$ has dimension $c+1$. Then $w - w_{0} = (a + 1)q + (b+1)p + (c-1) qp$. If $a < p-1$ and $b < q-1$ then $\mathcal{O}_{w-w_{0}}$ has dimension $c$. Furthermore $V$ is injective because $w - w_{0}$ is not a multiple of $qp$. Hence $\spc{coker}(V)_{w}$ is $1$-dimensional. If $a = p-1$ and $b < q-1$, then $w-w_{0} = (b+1)p + cqp$, so $\mathcal{O}_{w-w_{0}}$ has dimension $c+1$, $V$ is injective, and hence $\spc{coker}(V)_{w} = 0$. The same argument applies to the case $a < p-1$ and $b = q-1$. The only remaining case is $a = p-1$ and $b=q-1$. In this case $w-w_{0} = (c+1)qp$ and $\mathcal{O}_{w-w_{0}}$ has dimension $c+2$. But now $V$ has a $1$-dimensional kernel and so $\spc{coker}(V)_{w} = 0$. 

The upshot is that the cokernel is non-zero precisely when $a < p-1$ and $b < q-1$, in which case it is $1$-dimensional. These dimensions match with the dimensions of $\mathbb{C}[f] \otimes_{\mathbb{C}}C$. Hence it suffices for us to prove that $f^{c} x^{a} y^{b}$ is not in the image of $V$. We will do this by proving that the following map 
\[
M : \mathbb{C} \oplus \mathcal{O}_{w-w_{0}} \to \mathcal{O}_{w}, \qquad (\lambda, g) \mapsto \lambda f^{c}x^{a}y^b + V(g)
\]
is represented by a matrix with positive determinant, using the bases of Lemma \ref{weightspacedecompV}. Applying $V$ to the element $x^{a+1 + ip}y^{b+1+jq}$ yields 
\[
p(1+b+jq)x^{a + (i+1)p}y^{b+jq} + q(1+a+ip)x^{a+ip}y^{b+(j+1)q}. 
\]
The salient thing to note is that the basis elements are consecutive and the coefficients are positive. Hence $V$ is represented by a $(c+1) \times c$ matrix such that column $i$ has positive entries in rows $i$ and $i+1$ and $0$ for the remaining rows. Using the binomial theorem, $f^{c}x^a y^b  = \sum_{k=0}^{c} (-1)^{k}x^{a + (c-k)p} y^{b + kq}$. The salient point here is that the terms are non-zero with alternating signs. These give the entries of the first column of the matrix $M$. Computing the determinant of $M$ using the Laplace expansion along the first column shows that it is positive. 
\end{proof}

\subsection*{The dgla $(U_{0}, \delta_{S})$}
Now we choose a Lie algebra $\mathfrak{g}$ and a semisimple element $S \in \mathfrak{g}$. This induces an eigenspace decomposition of the Lie algebra 
\[
\mathfrak{g} = \bigoplus_{\lambda} \mathfrak{g}_{\lambda},
\]
where $\mathfrak{g}_{\lambda}$ is the eigenspace of $ad_{S}$ with eigenvalue $\lambda$. We will use the following convention: if $\lambda$ is not an eigenvalue of $ad_{S}$, then $\mathfrak{g}_{\lambda} = 0$. Note that the decomposition is preserved by the bracket: $[\mathfrak{g}_{\lambda}, \mathfrak{g}_{\mu}] \subseteq \mathfrak{g}_{\lambda + \mu}$. 

The dgla $(U_{0}, \delta_{S})$ has terms in degrees $0$ and $1$. They are given by 
\[
U^{0}_{0} = \bigoplus_{w \geq 0} \mathcal{O}_{w} \otimes \mathfrak{g}_{-w}, \qquad U^{1}_{0} = \bigoplus_{w \geq 0} \mathcal{O}_{w} \otimes \mathfrak{g}_{w_{0} - w} \beta,
\]
with $\delta_{S}$ given by applying $V$. We will sometimes drop $\beta$ from the notation. 

Applying Lemmas \ref{kernelofV} and \ref{cokernelofV} we obtain the following description of the cohomology of $(U_{0}, \delta_{S})$. 

\begin{corollary} \label{quasiisocohomology}
The cohomology of $(U_{0}, \delta_{S})$ is given as follows 
\[
H^{0}(U_{0}) = \bigoplus_{c \geq 0} f^{c} \mathfrak{g}_{-cpq}, \qquad H^{1}(U_{0}) \cong \bigoplus_{w \geq 0} (\mathbb{C}[f] \otimes C)_{w} \otimes \mathfrak{g}_{w_{0} - w}. 
\]
Furthermore, the graded Lie algebra $H^{\bullet}(U_{0})$ with zero differential naturally embeds into $(U_{0}, \delta_{S})$ as a quasi-isomorphic sub-dgla. 
\end{corollary}
Let $a \in \mathfrak{g}$ be a real semisimple element. Recall from \cite{boalch2011riemann} that this determines a parabolic subgroup of $G$ 
\[
P(a) = \{ g \in G \ | \ \lim_{z \to 0} z^{a} g z^{-a} \text{ exists in $G$ along any ray} \},
\]
where $z \in \mathbb{C}$ and $z^{a} = \exp(\log(z) a)$. Decomposing $S$ into real and imaginary parts, $S = a + i b$, we can define the following subgroup of $G$
\[
P_{S} := C_{G}(e^{\frac{-2\pi i}{qp}S}) \cap P(\frac{-a}{qp}).
\]
In this definition $C_{G}(e^{\frac{-2\pi i}{qp}S})$ is the centralizer of $e^{\frac{-2\pi i}{qp}S}$ in $G$. It is reductive but possibly disconnected. Let $C_{S}$ denote the connected component of the identity. The group $P_{S}$ is the parabolic subgroup of $C_{G}(e^{\frac{-2\pi i}{qp}S})$ (or $C_{S}$) determined by the element $\frac{-a}{qp}$ and it is connected. The reductive quotient of $P_{S}$ is $G_{S}$, the centralizer of $S$ in $G$. Denote the quotient map $\chi: P_{S} \to G_{S}$. 

\begin{lemma} \label{describingparabolic}
The group $P_{S}$ embeds into $Aut(S)$ as the subgroup integrating $H^{0}(U_{0})$. The gauge action of $P_{S}$ preserves $H^{1}(U_{0}) \subset U^{1}_{0}$ and is linear. Hence, we have a Lie subgroupoid 
\[
P_{S} \ltimes H^{1}(U_{0}) \subseteq Aut(S) \ltimes U^{1}_{0}. 
\]
\end{lemma}
\begin{proof}
Let $\mathbb{C} \ltimes \mathbb{C}^2$ be the Lie algebroid generated by the action of $E$ and let $\mathbb{C} \ltimes \mathbb{C}$ be the Lie algebroid generated by the action of $z \partial_{z}$. The following defines a Lie algebroid morphism 
\[
f: \mathbb{C} \ltimes \mathbb{C}^2 \to \mathbb{C} \ltimes \mathbb{C}, \qquad (\lambda, x, y) \mapsto (pq \lambda, f(x,y)),
\]
and under this map, the logarithmic connection $d + \frac{1}{qp}S d \log z$ pulls back to $p^* S$. As a result, the pullback defines an embedding of automorphism groups from $Aut(d + \frac{1}{qp}S d \log z) \to Aut(S)$. In \cite[Proposition 3.4]{bischoff2020lie} it is shown that restricting an automorphism to $1 \in \mathbb{C}$ defines an embedding of $Aut(d + \frac{1}{qp}S d \log z)$ into $G$ which identifies it with $P_{S}$. Furthermore, the Lie algebra of $P_{S}$ is identified with $\bigoplus_{c \geq 0} z^{c} \mathfrak{g}_{-cpq}$, and under the pullback, this is sent isomorphically to $H^{0}(U_{0})$. Finally, since the action of $H^{0}(U_{0})$ preserves $H^{1}(U_{0})$ and is linear, the same is true of $P_{S}$. 
\end{proof}

Given the semisimple element $S \in \mathfrak{g}$, we say that it is \emph{large enough} if all the positive integer eigenvalues of $ad_{S}$ are strictly greater than $w_{0}$. 

\begin{proposition} \label{planecurveME1}
The inclusion $P_{S} \ltimes H^{1}(U_{0}) \subseteq Aut(S) \ltimes U^{1}_{0}$ is a Morita equivalence if $S$ is large enough. 
\end{proposition}
\begin{proof}
First, because of the assumption on $S$ and the fact that $\mathcal{O}_{w_{0}} = 0$ (see Lemma \ref{weightspacedecompV}), the vector space $U^{1}_{0}$ has the following form 
\[
U^{1}_{0} = \bigoplus_{w > w_{0}} \mathcal{O}_{w} \otimes \mathfrak{g}_{w_{0} - w} \beta.
\]
We now proceed in several steps. 
\begin{enumerate}
\item Claim: The subspace $H^{1}(U_{0})$ intersects every orbit of $Aut(S)$. Given $\gamma \in U^{1}_{0}$, we need to find an element of $Aut(S)$ which sends $\gamma$ into $H^{1}(U_{0})$. We do this iteratively following the usual proof of the normal form for ODEs with Fuchsian singularities. First, we expand $\gamma = \sum_{w > w_{0}} \gamma_{w}$, where $\gamma_{w} \in \mathcal{O}_{w} \otimes \mathfrak{g}_{w_{0} - w} \beta$. Given a weight $w' > w_{0}$, let $u \in \mathcal{O}_{w' - w_{0}} \otimes \mathfrak{g}_{w_{0} - w'}$, and consider the action of $e^{u} \in Aut(S)$ on $\gamma$: 
\[
e^{u} \ast \gamma = e^{u} \gamma e^{-u} - V(e^{u})e^{-u} \beta. 
\]
We claim that $\gamma$ is modified in weights $w'$ and higher. Indeed, expanding we get
\[
V(e^{u})e^{-u} = V(u) - V(u)u + \frac{1}{2}V(u^2) + ...
\]
The first term has weight $w'$. All other terms have higher weights since $w' - w_{0} > 0$. Expanding the term $e^{u} \gamma_{w} e^{-u}$ gives 
\[
exp(ad_{u})\gamma_{w} = \gamma_{w} + [u, \gamma_{w}] + \frac{1}{2}[u,[u,\gamma_{w}]] + ...
\]
The second term has weight $w' - w_{0} + w > w'$, since $w > w_{0}$. Note that the action on weight $w'$ is given by $\gamma_{w'} \mapsto \gamma_{w'} - V(u) \beta$. 

By Lemma \ref{cokernelofV}, the element $\gamma_{w'}$ can be decomposed as 
\[
\gamma_{w'} = f^{c}x^{a}y^b \otimes X \beta + V(u) \beta,
\]
where $f^{c}x^{a}y^b \otimes X  \in (\mathbb{C}[f] \otimes C)_{w'} \otimes \mathfrak{g}_{w_{0} - w'}$ and $u \in \mathcal{O}_{w' - w_{0}} \otimes \mathfrak{g}_{w_{0} - w'}$. Then $(e^{u} \ast \gamma)_{w'} = \gamma_{w'} - V(u) \beta \in H^{1}(U_{0})$. 

Now starting with the lowest weight $w' > w_{0}$, we iteratively act on $\gamma$ by elements $e^{u} \in Aut(S)$ so that the terms up to level $w'$ lie in $H^{1}(U_{0})$. This will terminate after finitely many steps since $U^{1}_{0}$ is finite-dimensional. The result is an element of $H^{1}(U_{0})$. 
\item Claim: The inclusion functor $P_{S} \ltimes H^{1}(U_{0}) \to Aut(S) \ltimes U^{1}_{0}$ is fully-faithful. We need to show that given $g \in Aut(S)$ and $\gamma \in H^{1}(U_{0})$, if $g \ast \gamma \in H^{1}(U_{0})$, then $g \in P_{S}$. Recall the Levi decomposition $Aut(S) \cong Aut_{0}(S) \rtimes G_{S}$, and note that $G_{S} \subseteq P_{S}$. It therefore suffices to work under the assumption that $g \in Aut_{0}(S)$. Since such a $g$ is unipotent, it has the form $g = e^{u}$, for $u \in \bigoplus_{w > 0} \mathcal{O}_{w} \otimes \mathfrak{g}_{-w}$. Expanding in weights, $u = \sum_{w \geq w_{1}} u_{w}$, where $w_{1} > 0$ is the lowest weight. From the above expressions for $e^{u} \ast \gamma$, we see that the lowest weight for which $\gamma$ is modified is $w_{0} + w_{1}$. The corresponding term is given by 
\[
(e^u \ast \gamma)_{w_{0} + w_{1}} = \gamma_{w_{0} + w_{1}} + V(u_{w_{1}}). 
\]
Since $e^{u} \ast \gamma \in H^{1}(U_{0})$, we must have $V(u_{w_{1}}) = 0$, implying that $u_{w_{1}} \in H^{0}(U_{0})$ and $e^{u_{w_{1}}} \in P_{S}$. Let $\tilde{\gamma} = e^{u_{w_{1}}} \ast \gamma \in H^{1}(U_{0})$, so that $g \ast \gamma = (e^{u} e^{-u_{w_{1}}}) \ast \tilde{\gamma}$. Using the Baker-Campbell-Hausdorff formula and the fact that $w_{1} > 0$, we see that $e^{u} e^{-u_{w_{1}}} = e^{v}$, where the lowest weight of $v$ is strictly greater than $w_{1}$. Hence the result follows by induction on $w_{1}$. 

\item Let $\gamma \in H^{1}(U_{0})$. Claim: The inclusion $(H^{\bullet}(U_{0}), ad_{\gamma}) \to (U_{0}^{\bullet}, \delta_{S} + ad_{\gamma})$ is a quasi-isomorphism. This follows by the homological perturbation lemma \cite{crainic2004perturbation}. Let $a: H^{\bullet}(U_{0}) \to U^{\bullet}_{0}$ be the inclusion. By Corollary \ref{quasiisocohomology}, this is a quasi-isomorphism with respect to the $0$ differential on the domain and $\delta_{S}$ on the codomain. We have the decomposition $U^{1}_{0} = H^{1}(U_{0}) \oplus \spc{Im}(\delta_{S})$. Let $C$ be a complement to $H^{0}(U_{0})$, so that $U^{0}_{0} = H^{0}(U_{0}) \oplus C$. It is possible to choose this complement compatible with the weight decomposition. Using the decomposition we define the projection $b : U^{\bullet}_{0} \to H^{\bullet}(U_{0})$. The restriction $\delta_{S}|_{C} : C \to \spc{Im}(\delta_{S})$ is an isomorphism, and the inverse defines a map $h: U^{1}_{0} \to U^{0}_{0}$ which has weight $-w_{0}$. These maps satisfy $b \circ a = id$, $id - a \circ b = [ \delta_{S}, h]$, as well as the side conditions $h \circ a = 0$, $b \circ h = 0$ and $h \circ h = 0$. 

Now consider the map $ad_{\gamma} : U^{0}_{0} \to U^{1}_{0}$ which will serve as a perturbation. Note that it restricts to a map $H^{0}(U_{0}) \to H^{1}(U_{0})$. Expanding in the weights, $\gamma = \sum_{w > w_{0}} \gamma_{w}$. Since the weight of $h$ is $-w_{0}$, it follows that $ad_{\gamma} \circ h$ is an endomorphism of $U^{1}_{0}$ which raises the weight of an element by at least $1$. It follows that $ad_{\gamma} \circ h$ is nilpotent, allowing us to apply the perturbation lemma. Using the formulas appearing above Lemma \ref{perthmtpy}, we see that $a$ remains unperturbed, $\delta_{S}$ is deformed to $\delta_{S} + ad_{\gamma}$ and the zero differential is deformed to $ad_{\gamma}$. 

\end{enumerate}
\end{proof}

\subsection*{The moduli stack $[\mathcal{W}(A)/Aut(S)]$}
Now we choose an element $A = S + N_{0} \in \mathfrak{g}$, which we write using the Jordan decomposition. This determines the derived stack $[\mathcal{W}(A)/Aut(S)]$. The base of the derived manifold is 
\[
W(A) = \{ C \beta + N \alpha_{0} \in U^{1}_{0} \oplus U^{0}_{0} \alpha_{0} \ | \ N(0) \in G_{S} \ast N_{0} \}. 
\]
The bundle of curved dgla's is the trivial bundle $W(A) \times U^{1}_{0} \alpha_{0}$ with trivial dgla structure. The curvature section is given by $F(C \beta + N \alpha_{0}) = (V(N) + [C, N] ) \beta \wedge \alpha_{0}$. Applying Lemma \ref{describingparabolic} and Proposition \ref{planecurveME1} we can construct a smaller model for this derived stack.

By Lemma \ref{describingparabolic}, the vector space $H^{1}(U_{0})$ is a linear representation of $P_{S}$. The Lie algebra $Lie(P_{S})$ is likewise a representation, and the subspace $d\chi^{-1}(G_{S} \ast N_{0})$ is preserved by this action (recall that $\chi: P_{S} \to G_{S}$ is the projection to the reductive quotient). Let \[ Q(A) = H^{1}(U_{0}) \times d\chi^{-1}(G_{S} \ast N_{0}),\] 
equipped with the action of $P_{S}$. The infinitesimal action of $P_{S}$ on $H^{1}(U_{0})$ defines a $P_{S}$-equivariant map 
\[
F_{S} : Q(A) \to H^{1}(U_{0}), \qquad (C, N) \mapsto [C,N]. 
\]
Viewing this as a section of the bundle $Q(A) \times H^{1}(U_{0})$ we get a derived manifold $\mathcal{Q}(A)$, which represents the derived vanishing locus of $F_{S}$. This defines the derived stack 
\[
[\mathcal{Q}(A)/P_{S}]. 
\]
By Corollary \ref{quasiisocohomology} and Lemma \ref{describingparabolic} this maps into $[\mathcal{W}(A)/Aut(S)]$. 

\begin{theorem} \label{parabolicderivedstack}
There is a map of derived stacks
\[
i: [\mathcal{Q}(A)/P_{S}] \to [\mathcal{W}(A)/Aut(S)].
\]
For points $(0, N) \in Q(A)$ the derivative $di_{(0,N)}$ is a quasi-isomorphism of tangent complexes. Furthermore, if $S$ is large enough, then $i$ is an equivalence. 
\end{theorem}
\begin{proof}
To begin, assume that $S$ is large enough, so that $P_{S} \ltimes H^{1}(U_{0}) \subseteq Aut(S) \ltimes U^{1}_{0}$ is a Morita equivalence by Proposition \ref{planecurveME1}. Now let $C \beta + N \alpha_{0} \in MC(\mathcal{W}(A))$. Claim: If $C \in H^{1}(U_{0})$, then $(C,N) \in MC(\mathcal{Q}(A))$. Indeed, observe that $F(C \beta + N \alpha_{0}) = (\delta_{S}(N) + ad_{C \beta}(N)) \wedge \alpha_{0}$, and that 
\begin{equation} \label{quasiisotruncated}
(H^{\bullet}(U_{0}), ad_{C \beta}) \to (U_{0}^{\bullet}, \delta_{S} + ad_{C \beta})
\end{equation}
is a quasi-isomorphism. It follows that $N \in H^{0}(U_{0})$, and therefore that the claim is verified. It is straightforward to deduce from this claim that the induced morphism 
\[
P_{S} \ltimes MC(\mathcal{Q}(A)) \to Aut(A) \ltimes MC(\mathcal{W}(A))
\]
is an equivalence of categories. Now given a point $(C, N) \in MC(\mathcal{Q}(A))$, we need to show that the morphism of tangent complexes $di_{(C,N)} :  \mathbb{T}_{(C,N)} [\mathcal{Q}(A)/P_{S}]  \to  \mathbb{T}_{(C,N)}  [\mathcal{W}(A)/Aut(S)] $ is a quasi-isomorphism. The differentials on these tangent complexes have the form $d + ad_{N \alpha_{0}}$. By an argument involving the perturbation lemma, it suffices to prove that $di_{(C,N)}$ is a quasi-isomorphism with respect to the differentials $d$. But for these differentials, $di_{(C,N)}$ is a direct sum of Equation \ref{quasiisotruncated} and a shift of a subspace. Therefore, it is a quasi-isomorphism. Note that by Corollary \ref{quasiisocohomology}, Equation \ref{quasiisotruncated} is quasi-isomorphism when $C = 0$ even if $S$ is not large enough. 
\end{proof}
\begin{remark}
Connections of the form $(0,N) \in Q(A)$ are pullbacks by $f$ of connections on $\mathbb{C}$ with a logarithmic pole at the origin. 
\end{remark}

The condition in Theorem \ref{parabolicderivedstack} that $S$ is large enough is necessary. In the following Example we see that the moduli space can have extra components when the condition is not satisfied. 
\begin{example} \label{Extracomponent}
Let $f = x^2 - y^5$, let $\mathfrak{g} = \mathfrak{gl}_{3}$, and let $A = S$ be a diagonal matrix with entries $0, 1$ and $11$. Note that $S$ is not large enough since $1$ is a positive eigenvalue of $ad_{S}$ which is smaller than $w_{0} = 3$. The subalgebra $\mathfrak{g}_{S}$ consists of the diagonal matrices and 
\begin{align*}
U_{0}^{0} &= \mathfrak{g}_{S} \oplus \mathrm{span}_{\mathbb{C}}(x^2 E_{23}, y^5 E_{23}, xy^3 E_{13}), \\ 
U_{0}^{1} &= \mathrm{span}_{\mathbb{C}}(yE_{21}, y^2 E_{12}, xy^4 E_{23}, x^2 y^2 E_{13}, y^7 E_{13}),
\end{align*}
where $E_{ij}$ are the elementary matrices. On the other hand, the $\delta_{S}$ cohomology is given by 
\[
H^0(U_{0}) =  \mathfrak{g}_{S} \oplus \mathrm{span}_{\mathbb{C}}(fE_{23}), \ \ H^{1}(U_{0}) = \mathrm{span}_{\mathbb{C}}(yE_{21}, y^2 E_{12}, y^2 f E_{13}).
\]

An arbitrary element of $W(A)$ has the form $C \beta + N \alpha_{0}$, where
\begin{align*}
C &= C_{21}yE_{21} + C_{12}y^2 E_{12} + C_{23}xy^4E_{23} + (C_{13}^{(1)} x^2y^2 + C_{13}^{(2)}y^7)E_{13}, \\
N &= N_{13}xy^3 E_{13} + (N_{23}^{(1)} f + N_{23}^{(2)}(x^2 + y^5))E_{23}. 
\end{align*}
The Maurer-Cartan equation consists of the following coupled system of equations: 
\begin{align*}
20 N_{23}^{(2)} &= - C_{21}N_{13} \\ 
5 N_{13} &=- C_{12}(N_{23}^{(2)} - N_{23}^{(1)}) \\ 
6N_{13} &= -C_{12} (N_{23}^{(2)} + N_{23}^{(1)}). 
\end{align*}
Adding the last two equations and substituting the result into the first yields $(110 - C_{21}C_{12}) N_{23}^{(2)} = 0$. Assume first that $C_{21}C_{12} \neq 110$. Then the equations simplify to $N_{23}^{(2)} = N_{13} = C_{12}N_{23}^{(1)} = 0$, and $N \in H^0(U_{0})$. The result is a $5$-dimensional variety $M_{a}$ with $2$ irreducible components. Next, assume that $C_{21}C_{12} = 110$. Then we can solve the equations to obtain $N_{13} = -2 C_{12} N_{23}^{(1)}$ and $N_{23}^{(2)} = 11 N_{23}^{(1)}$. Hence, the result is a smooth irreducible $5$-dimensional variety $M_{b}$. 

An element of $Aut(S)$ has the form 
\[
g = \begin{pmatrix} u & 0 & \lambda xy^3 \\ 0 & v & ax^2 + b y^5 \\ 0 & 0 & w \end{pmatrix},
\]
and it acts on $C = (C_{21}, C_{12}, C_{23}, C_{13}^{(1)}, C_{13}^{(2)}) $ and $N = (N_{13}, N_{23}^{(1)}, N_{23}^{(2)})$ in the following way: 
\begin{align*}
g \ast C &= (\frac{v}{u}C_{21}, \frac{u}{v}C_{12}, \frac{v}{w}C_{23} - \frac{\lambda v}{uw}C_{21} - \frac{10 (a+ b)}{w}, \frac{u}{w}C_{13}^{(1)} - \frac{a u}{vw}C_{12} - \frac{6 \lambda }{w}, \frac{u}{w}C_{13}^{(2)} - \frac{b u}{vw}C_{12} - \frac{5 \lambda }{w}) \\
g \ast N &= (\frac{u}{w}N_{13}, \frac{v}{w}N_{23}^{(1)}, \frac{v}{w}N_{23}^{(2)}). 
\end{align*}
There are two things we can immediately note. First, by looking at the action on $N$, we can see that there are orbits in $M_{b}$ which do not intersect $Q(A)$. Hence, $[M_{b}/Aut(S)]$ is an extra component of the moduli space. Second, the product $C_{21}C_{12}$ is invariant under the action and if we assume that $C_{21}C_{12} \neq 110$, then it is always possible to send $C$ into $H^{1}(U_{0})$ (i.e. $C_{23} = 0$ and $C_{13}^{(1)} + C_{13}^{(2)} = 0$). The subgroup of $Aut(S)$ which stabilizes this locus is $P_{S}$, and hence $[M_{a}/Aut(S)]$ is contained in $[MC(\mathcal{Q}(A))/P_{S}]$. 
\end{example}

\subsection*{Relation to the Grothendieck-Springer resolution}
Recall that $P_{S} \subseteq C_{S}$ is a parabolic subgroup with Levi factor $G_{S}$. Denote their Lie algebras $\frak{p}_{S}$, $\frak{c}_{S}$ and $\frak{g}_{S}$, respectively. Let $\mathbb{P}$ denote the partial flag variety, defined as the set of parabolic subalgebras of $\frak{c}_{S}$ which are conjugate to $\frak{p}_{S}$. It is isomorphic to $C_{S}/P_{S}$. Consider the incidence variety
\[
\tilde{\frak{c}} = \{ (x, \mathfrak{p}) \in \frak{c}_{S} \times \mathbb{P} \ | \ x \in \mathfrak{p} \},
\]
which is isomorphic to $(C_{S} \times \frak{p}_{S})/P_{S}$ via the map which sends $(g, x) \in C_{S} \times \frak{p}_{S}$ to the pair $(Ad_{g}(x), Ad_{g}(\frak{p}_{S})) \in \tilde{\frak{c}} $. When $P_{S}$ is a Borel subgroup, $\tilde{\frak{c}}$ is the Grothendieck-Springer resolution (see e.g. \cite{MR2838836} for details). The element $N_{0} \in \mathfrak{g}_{S}$ defines an adjoint orbit $\mathcal{O}$ in the Levi quotient of every parabolic $\frak{p} \in \mathbb{P}$ \cite[Lemma 1]{boalch2011riemann}. This lets us define the following subspace of $\tilde{\frak{c}} $: 
\[
\tilde{\frak{c}}_{N_{0}} = \{ (x, \mathfrak{p}) \in \tilde{\frak{c}} \ | \ d\chi(x) \in \mathcal{O} \},
\]
where $d\chi$ denotes the projection to the Levi quotient. The space $Q(A)$ is a $P_{S}$-equivariant vector bundle over $d\chi^{-1}(G_{S} \ast N_{0})$. Hence $\pi: E_{A} = (C_{S} \times Q(A))/P_{S} \to \tilde{\frak{c}}_{N_{0}}$ is a $C_{S}$-equivariant vector bundle and the map $F_{S}$ gives rise to an equivariant section $\sigma_{S}$ of $\pi^{*}(E_{A}) \to E_{A}$. In this way, we obtain a derived stack $[\mathcal{E}_{A}/C_{S}]$ which is equivalent to $[\mathcal{Q}(A)/P_{S}]$. 

\begin{example} \label{flagvarietyexample}
Let $G = GL_{n}$ and let $A = S = pq D$, where $D$ is a diagonal matrix with distinct integer eigenvalues. In this case, $S$ is large enough, $C_{S} = G$ and $P_{S} = B$ is a Borel subgroup. Hence $\mathbb{P}$ is the flag variety $Fl_{n}$. Since $N_{0} = 0$, $\tilde{\frak{g}}_{0}$ is the Springer resolution, which is isomorphic to $T^* Fl_{n}$. We use Corollary \ref{quasiisocohomology} to compute the cohomology space $H^{1}(U_{0})$. The weights $w$ showing up in the decomposition have the form $w = (p-1)q + (q-1)p + cqp$. But $ (\mathbb{C}[f] \otimes C)_{w} = 0$ in this case, so that $H^{1}(U_{0}) = 0$. Hence, the moduli space $[\mathcal{W}(A)/Aut(S)]$ is equivalent to the quotient stack 
\[
[T^* Fl_{n}/GL_{n}]. 
\]
Note that the connections corresponding to the points of $T^* Fl_{n}$ are all pulled-back from logarithmic connections on $\mathbb{C}$, where a similar classification is given by \cite[Theorem A]{boalch2011riemann}. 
\end{example}

\begin{example} 
Let $G = GL_{n}$ and $A = S = \frac{qp}{r}D$, where $D$ is a diagonal integer matrix with distinct eigenvalues and $r$ is a positive integer. In this case, the eigenvalues of $e^{\frac{-2\pi i}{qp}S}$ are $r^{th}$ roots of unity, implying that $C_{S} = \prod_{i = 0}^{r-1} GL_{m_{i}}$, with the factors indexed by the roots of unity. Since the eigenvalues of $S$ are assumed to be distinct, $P_{S} = \prod_{i = 0}^{r-1} B_{m_{i}}$ is a product of Borels. Therefore, $\mathbb{P} \cong  \prod_{i = 0}^{r-1} Fl_{m_{i}}$ is a product of flag varieties and $\tilde{\mathfrak{c}}_{0} \cong \prod_{i = 0}^{r-1} T^*Fl_{m_{i}}$ is the product of their cotangent bundles. Writing the eigenvalues of $D$ as $rk + u$, with $0 \leq u < r$, the eigenvalues of $ad_{S}$ have the form $qp(k_{1} - k_{2}) + \frac{qp}{r}(u_{1}-u_{2})$. We can ensure that $S$ is large enough by restricting the possible values of $k_{i}$. For example, this is guaranteed if $k_{1} - k_{2} \neq 0, \pm 1$ for all pairs of eigenvalues. By Corollary \ref{quasiisocohomology}, the weights $w$ showing up in the decomposition of $H^{1}(U_{0})$ have the form 
\[
w = qp(1 + k_{2} - k_{1}) + \frac{qp}{r}(u_{2} - u_{1}) -p - q. 
\]
As in Example \ref{flagvarietyexample}, we must have $u_{2} \neq u_{1}$ in order to get a non-zero contribution. 

Now we specialise to $GL_{4}$, with $r = pq$ and $S = D$ a diagonal matrix with entries \[ pqk_{1}, \ \  pqk_{2}, \ \  pqk_{3} + p + q, \ \ pqk_{4} + p + q, \] such that $k_{1} \ll k_{2} \ll k_{3} \ll k_{4}$. Then $C_{S} = GL_{2} \times GL_{2}$ and $\mathbb{P} \cong \mathbb{P}^1 \times \mathbb{P}^1$. Let $T_{1}$ and $T_{2}$ be the tautological rank $2$ vector bundles over $\mathbb{P}^1 \times \mathbb{P}^1$ (corresponding respectively to the first and second factors). They are both trivial, but are equipped with tautological line subbundle bundles $L_{i}$ with respective degrees $(-1,0)$ and $(0,-1)$. The cotangent bundle $\tilde{\mathfrak{c}}_{0}$ can be identified with $Hom(T_{1}/L_{1}, L_{1}) \oplus Hom(T_{2}/L_{2}, L_{2})$, or alternatively, the nilpotent filtration-preserving endomorphisms of $T_{1}$ and $T_{2}$. 

The weights in the decomposition of $H^{1}(U_{0})$ have the form $qp(1 + k_{j} - k_{i})$, for $j = 3,4$ and $i = 1, 2$. Hence, $H^{1}(U_{0})$ can be identified with the subspace of $\mathfrak{gl}_{4}$ consisting of the upper right $2 \times 2$ block. Therefore, 
\[
E_{A} = Hom(T_{1}/L_{1}, L_{1}) \oplus Hom(T_{2}/L_{2}, L_{2}) \oplus Hom(T_{2}, T_{1}),
\]
with section $\sigma_{S}(a,b,c) = c \circ b - a \circ c$. 
\end{example}

\subsection*{Tangent Lie bialgebra and shifted Poisson geometry}
Let $A = S \in \mathfrak{g}$ be a semisimple element and consider the moduli space $[\mathcal{W}(S)/Aut(S)]$. There is a distinguished point $0 \in MC(\cal{W}(S))$ corresponding to the connection $1$-form $\alpha_{0}S$. In this section we focus our attention on a formal neighbourhood of this point in the moduli space and sketch the construction of a $-2$-shifted Poisson structure on this neighbourhood. According to the fundamental principal of derived deformation theory (see e.g. \cite{kontsevich2002deformation, MR4485797, MR2628795, MR2827833}) this formal neighbourhood is encoded by the shifted tangent complex $\mathbb{T}_{0}[\mathcal{W}(S)/Aut(S)][-1]$, equipped with its structure as a differential graded Lie algebra. By Theorem \ref{parabolicderivedstack} this dgla is quasi-isomorphic to 
\[
\mathbb{T}_{0}[\mathcal{Q}(A)/P_{S}][-1] = H^{0}(U_{0}) \to H^{1}(U_{0}) \oplus H^{0}(U_{0})_{+} \to H^{1}(U_{0}),
\]
where the differential is $0$ and $H^{0}(U_{0})_{+} = \bigoplus_{c > 0} f^{c} \mathfrak{g}_{-cpq}$. By \cite[Proposition 1.5]{pimenov2015shifted}, a $-3$-shifted Lie bialgebra structure on $\mathbb{T}_{0}[\mathcal{Q}(A)/P_{S}][-1]$ gives rise to a $-2$-shifted Poisson structure on the formal neighbourhood. Hence, our strategy is to construct a Lie bialgebra structure on the tangent complex by embedding it into $H^{\bullet}(U_{0}) \ltimes H^{\bullet}(U_{0})[-1]$, and then to realize this larger Lie algebra as a Lagrangian inside a $-3$-shifted Manin triple. 

We construct the Manin triple in stages, starting with the following input data: 
\begin{itemize}
\item Choose an invariant inner product $k$ on the reductive Lie algebra $\mathfrak{g}$. This induces a perfect pairing between the eigenspaces $\mathfrak{g}_{\lambda}$ and $\mathfrak{g}_{-\lambda}$. 
\item The $\mathbb{C}$-vector space $C = \mathbb{C}[x,y]/(x^{p-1}, y^{q-1})$ is canonically isomorphic to the degree $1$ cohomology of $f^{-1}(1)$, which is a curve of genus $g = \frac{1}{2}(p-1)(q-1)$ with a single puncture. The isomorphism is given as follows: 
\[
C \to H^{1}(f^{-1}(1)), \qquad x^a y^b \mapsto (x^a y^b \beta)|_{f^{-1}(1)}. 
\]
By pulling back the intersection pairing on the curve, we obtain a symplectic form $I$ on $C$. Up to a scaling constant, it is given by the following formula
\[
I(x^a y^b, x^{a'} y^{b'}) = \frac{\delta(a + a' + 2 - p) \delta(b + b' + 2 - q)}{aq + bp - w_{0}},
\]
where $\delta$ is the delta function which evaluates to $1$ at $0$ and $0$ otherwise. We can extend this to a $\mathbb{C}[f]$-linear pairing on $\mathbb{C}[f] \otimes_{\mathbb{C}} C$. 
\end{itemize}

Now consider the Lie algebra $\mathfrak{c} = \bigoplus_{c \in \mathbb{Z}} f^{c} \mathfrak{g}_{-cpq}$. This has two distinguished subalgebras: first $\mathfrak{b} = H^{0}(U_{0})$, where $c \geq 0$, and second $\mathfrak{b}^{-}$, where $c \leq 0$. Note that $\mathfrak{c}$ can be viewed as a subalgebra of $\mathfrak{g}$ and so inherits the pairing $k$. This defines a perfect pairing between $\mathfrak{b}$ and $\mathfrak{b}^{-}$.

Next, define the following vector space
\[
\mathfrak{K} = \bigoplus_{c \in \mathbb{Z}} \bigoplus_{\substack{0 \leq a \leq p-2 \\ 0 \leq b \leq q-2} } f^{c}x^{a}y^{b} \mathfrak{g}_{w_{0}-cpq - aq - bp}.
\]
This decomposes as $\mathfrak{K} = \mathfrak{n}_{+} \oplus \mathfrak{h} \oplus \mathfrak{n}_{-}$ according to whether $c$ is positive, zero, or negative, respectively. Note that $H^{1}(U_{0}) = \mathfrak{n}_{+} \oplus \mathfrak{h}$. Viewing both $\mathfrak{c}$ and $\mathfrak{K}$ as subspaces of the Lie algebra $\mathbb{C}[x,y]\otimes \mathfrak{g}$, we can show that $\mathfrak{K}$ is a representation of $\mathfrak{c}$. By combining the symplectic form $I$ on $C$ with the Lie bracket on $\mathfrak{g}$, we define a symmetric $\mathfrak{c}$-equivariant map $\omega: S^{2}(\mathfrak{K}) \to \mathfrak{c}$ as follows
\[
\omega(f^{c_{1}}x^{a_{1}} y^{b_{1}} X_{1}, f^{c_{2}}x^{a_{2}} y^{b_{2}} X_{2}) = f^{c_{1}+c_{2}} I(x^{a_{1}} y^{b_{1}}, x^{a_{2}}y^{b_{2}}) [X_{1}, X_{2}]. 
\]
Post-composing this with the natural projection to $\mathfrak{b}^{-}$ defines a map $\mu : S^{2}(\mathfrak{K}) \to \mathfrak{b}^{-}$. This is $\mathfrak{b}$-equivariant, where $\mathfrak{b}^{-}$ is a $\mathfrak{b}$-representation by using the inner product to identify it with $\mathfrak{b}^{*}$. With respect to the $\mathfrak{b}$ action on $\mathfrak{K}$, both $\mathfrak{n}_{+}$ and $\mathfrak{n}_{+} \oplus \mathfrak{h}$ are sub-representations. Hence the quotient $\mathfrak{h} = \frac{\mathfrak{n}_{+} \oplus \mathfrak{h}}{\mathfrak{n}_{+}}$ is naturally a $\mathfrak{b}$-representation, and $\mu$ descends to define a $\mathfrak{b}$-equivariant map $\nu : S^{2}(\mathfrak{h}) \to \mathfrak{b}^{-}$. We now define a graded Lie algebra 
\[
L = \mathfrak{b} \oplus (\mathfrak{K} \oplus \mathfrak{h})[-1] \oplus \mathfrak{b}^{-}[-2]. 
\]
The bracket is constructed from the bracket on $\mathfrak{b}$, the $\mathfrak{b}$-action on the other summands, the symmetric pairing $\mu$ on $\mathfrak{K}$ and the symmetric pairing $-\nu$ on $\mathfrak{h}$. We set the bracket between $\mathfrak{K}$ and $\mathfrak{h}$ to be zero. 
\begin{lemma}
The vector space $L$ with the bracket described above defines a graded Lie algebra. Let $\mathfrak{p}_{+} = \mathfrak{b} \oplus (\mathfrak{n}_{+} \oplus \mathfrak{h})[-1]$ and let $\mathfrak{p}_{-} = (\mathfrak{n}_{-} \oplus \mathfrak{h})[-1] \oplus \mathfrak{b}^{-}[-2]$. There are morphisms 
\begin{align*}
\mathfrak{p}_{+}  \to L, & \qquad (b,n,h) \mapsto (b, n + h, h, 0) \\
\mathfrak{p}_{-}  \to L, & \qquad (n,h, u) \mapsto (0, n + h, -h, u).
\end{align*}
These embed $\mathfrak{p}_{\pm}$ as complementary subalgebras of $L$. Furthermore, $\mathfrak{p}_{+}$ is isomorphic to $H^{\bullet}(U_{0})$. 
\end{lemma}

Next we construct an inner product on $L$. By combining the inner product $k$ on $\mathfrak{g}$ with the symplectic pairing $I$ on $C$, we define a skew symmetric pairing $\Omega: \wedge^2 \mathfrak{K} \to \mathbb{C}$ as follows
\[
\Omega(f^{c_{1}}x^{a_{1}}y^{b_{1}}X_{1}, f^{c_{2}}x^{a_{2}}y^{b_{2}}X_{2}) = f^{c_{1} + c_{2}}I(x^{a_{1}}y^{b_{1}}, x^{a_{2}}y^{b_{2}}) k(X_{1}, X_{2}). 
\]
This pairing is non-degenerate, restricts to a non-degenerate pairing on $\mathfrak{h}$ and defines a perfect pairing between $\mathfrak{n}_{\pm}$. We can now define a non-degenerate graded symmetric bilinear pairing 
\[
B : S^2(L) \to \mathbb{C}[-2].
\]
More precisely, given $y = (b_{1}, k_{1}, h_{1}, u_{1}), z = (b_{2}, k_{2}, h_{2}, u_{2}) \in L$, we set 
\[
B(y,z) = k(b_{1}, u_{2}) + k(u_{1}, b_{2}) + \Omega(k_{1}, k_{2}) - \Omega(h_{1}, h_{2}). 
\]
\begin{lemma}
The pairing $B$ is a non-degenerate invariant inner product on $L$. Furthermore, $\mathfrak{p}_{\pm}$ are complementary Lagrangian subalgebras. In other words, $(L, \mathfrak{p}_{+}, \mathfrak{p}_{-})$ is a $-2$-shifted Manin triple. 
\end{lemma}

Now let $\mathbb{C}[\epsilon]$ be the cdga generated by a degree $+1$ variable $\epsilon$ and let $tr: \mathbb{C}[\epsilon] \to \mathbb{C}[-1]$ be the trace map defined by sending the element $a + b \epsilon$ to $b$. Tensoring with $\mathbb{C}[\epsilon]$ defines a new triple of graded Lie algebras $(\mathbb{C}[\epsilon] \otimes L,  \mathbb{C}[\epsilon] \otimes \mathfrak{p}_{+}, \mathbb{C}[\epsilon]\otimes \mathfrak{p}_{-})$ and the bilinear form $B$ extends in the following way 
\[
B : S^2(\mathbb{C}[\epsilon] \otimes L) \to \mathbb{C}[-3], \qquad (fy, g z) \mapsto (-1)^{|y| |g|} tr(fg) B(y,z).
\] 
In this way, we obtain a $-3$-shifted Manin triple. Note that $\mathbb{C}[\epsilon] \otimes \mathfrak{p}_{+} \cong H^{\bullet}(U_{0}) \ltimes H^{\bullet}(U_{0})[-1]$. The tangent complex $\mathbb{T}_{0}[\mathcal{Q}(A)/P_{S}][-1]$ is isomorphic to the subalgebra $\mathfrak{b} \oplus (\mathfrak{n}_{+} \oplus \mathfrak{h}) \oplus \mathfrak{b}_{>0} \epsilon \oplus (\mathfrak{n}_{+} \oplus \mathfrak{h})\epsilon$, where $\mathfrak{b}_{>0} $ consists of the elements $f^{c} X$ with $c>0$. Let $M$ be the direct sum of this subalgebra with $\mathbb{C}[\epsilon]\otimes \mathfrak{p}_{-}$. Then $M$ is a coisotropic subalgebra of $L$ and $M^{\perp} = \mathfrak{g}_{0} \subset \mathfrak{b}^{-} \subset M$ is an isotropic ideal. It follows that $(M/M^{-1}, \mathbb{T}_{0}[\mathcal{Q}(A)/P_{S}][-1], \mathbb{C}[\epsilon]\otimes \mathfrak{p}_{-}/M^{\perp})$ is a $-3$-shifted Manin triple. Therefore, by \cite[Lemma 1.3]{pimenov2015shifted}, the tangent complex obtains a $-3$-shifted Lie bialgebra structure. 

\begin{theorem} \label{tangentLiebialgebra}
The tangent complex $\mathbb{T}_{0}[\mathcal{Q}(A)/P_{S}][-1]$ admits the structure of a $-3$-shifted Lie bialgebra. Therefore, the formal neighbourhood of $0 \in [\mathcal{W}(S)/Aut(S)]$ admits a $-2$-shifted Poisson structure. 
\end{theorem}

\bibliographystyle{plain}

 \bibliography{bibliography.bib}

\end{document}